\documentclass[11pt]{article}

\usepackage{amstext, amsmath,latexsym,amsbsy,amssymb}
\usepackage{enumitem}

\newtheorem{proposition}{Proposition}[section]
\newtheorem{remark}{Remark}[section]
\newenvironment{proof}{{\bf Proof\ }}{\QED\\}
\newtheorem{lemma}{Lemma}[section]

\numberwithin{equation}{section}
\newtheorem{theorem}{Theorem}[section]

\newcommand{\QED}{\hspace*{\fill}\rule{2.5mm}{2.5mm}}

\allowdisplaybreaks
\begin{document}
\title{Convergence to Equilibrium of Some Kinetic Models}
\author{\\ Minh-Binh Tran\\
 Basque Center for Applied Mathematics\\
 Mazarredo 14, 48009 Bilbao Spain\\
Email:   tbinh@bcamath.org}
\maketitle
\maketitle
\begin{abstract}
We introduce in this paper a new approach to the problem of the convergence to equilibrium for kinetic equations. The idea of the approach is to prove a 'weak' coercive estimate, which implies exponential or polynomial convergence rate. Our method works very well not only for hypocoercive systems in which the coercive parts are degenerate but also for the linearized Boltzmann equation.
\end{abstract}
{\bf Keyword}
{Goldstein-Taylor model, Boltzmann equation, hard potential, soft potential, rate of convergence to equilibrium. 
\section{Introduction}
In \cite{DesvillettesVillani:2001:OTG} and \cite{DesvillettesVillani:2005:OTG}, L. Desvillettes and C. Villani started the program about the trend to equilibrium for kinetic equations. Up to now, there are three classes of techniques to study the convergence to equilibrium. The first class of technique is the Lyapunov functional technique, which works for nonlinear equations. These techniques are developed in \cite{CaceresCarrilloGoudon:2003:ERL}, \cite{DesvillettesVillani:2001:OTG}, \cite{DesvillettesVillani:2005:OTG}, \cite{Desvillettes:2006:HTE}, \cite{FellnerNeumannSchmeiser:2004:CGE}. The second class of techniques is the pseudodifferential calculus, which works for linear hypoelliptic equations, developed in \cite{Herau:2006:HET}, \cite{EckmannHairer:2003:SPH}, \cite{HelfferNier:2005:HES}, \cite{HerauNier:2004:IHT}, \cite{Villani:2006:HDO}. The third class of techniques is developed by Yan Guo in \cite{Guo:2002:LPB}, which is in some sense an intermediate method between the two previous ones, which works for nonlinear kinetic equations in a close-to-equilibrium regime or the linearized versions of nonlinear kinetic equations. For a full discussion on this, we refer to the note \cite{Villani:2009:H}.
\\ Using the techniques developed in \cite{DesvillettesVillani:2001:OTG}, \cite{DesvillettesVillani:2005:OTG}, L. Desvillettes and F. Salvarani have investigated the speed of relaxation to equilibrium in the case of linear collisional models where the collision frequency is not uniformly bounded away from $0$. The two models that they considered are the non-homogeneous transport equation and the Goldstein-Taylor model
\begin{equation}
\label{DS-Intro}
\frac{\partial f}{\partial t}+v.\nabla f=\sigma(x)(\bar{f}-f),
\end{equation}
and 
\begin{equation}\label{G-TIntro}
\left\{ \begin{array}{ll} \frac{\partial u}{\partial t}+\frac{\partial u}{\partial x}=\sigma(x)(v-u),\vspace{.1in}\\
\frac{\partial v}{\partial t}-\frac{\partial v}{\partial x}=\sigma(x)(u-v).\end{array}\right.
\end{equation}
They prove that when $\sigma$ is greater than a positive polynomial and $\sigma$ belongs to $H^2$, one can get polynomial decays of the solutions toward the equilibrium points. However, the techniques used in the paper could not be extended to consider the case where the cross section $\sigma$ is $0$ on a set of strictly positive measure. A conjecture in this paper is to find explicit decay rates for these systems in wider classes of $\sigma$. In the same spirit of \cite{DesvillettesSalvarani:2009:ABD}, K. Aoki and F. Golse \cite{AokiGolse:2011:OSA} have studied the case of a collisionless gas enclosed in a vessel, where the surface is kept at a constant temperature, and they have investigated the convergence to equilibrium for such a system.
\\ We introduce a new approach to the problem of convergence toward equilibrium in the kinetic theory and resolve the conjecture by L. Desvillettes and F. Salvarani in \cite{DesvillettesSalvarani:2009:ABD} for Goldstein-Taylor and related models. We can relax the regularity property of $\sigma$ as well as the condition that $\sigma$ is greater than a positive polynomial and prove that the decay is exponential (see Theorems $\ref{TheoremTransportexponentialdecay1}$, $\ref{TheoremTransportexponentialdecay2}$). The main idea of our techniques is similar to the work of Haraux \cite{Haraux:1989:RSC}: in order to prove an exponential decay for the solution of the equation
\begin{equation}\label{Introe1}
\left\{ \begin{array}{ll} \frac{\partial f}{\partial t}+\mathcal{A}(f)=-\mathcal{K}(f), t\in\mathbb{R}_+,\vspace{.1in}\\
f(0)=f_0, \end{array}\right.
\end{equation}
we can study the following homogeneous equation with the same initial condition
\begin{equation}\label{Introe2}
\left\{ \begin{array}{ll} \frac{\partial g}{\partial t}+\mathcal{A}(g)=0, t\in\mathbb{R}_+,\vspace{.1in}\\
g(0)=f_0,\end{array}\right.
\end{equation}
and prove that the following observability inequality holds 
\begin{equation}\label{Introe3}\int_0^T<\mathcal{K}(g),g>dt~~~\geq C\|f_0\|^2.\end{equation}
A natural way of proving the exponential decay for the solutions of $(\ref{Introe1})$ is to prove that $\mathcal{K}$ is coercive 
$$<\mathcal{K}(g),g>\geq C\|g\|^2,$$
however this is not always true, especially in the case of Goldstein-Taylor and related models. The task of proving of the observability inequality $(\ref{Introe3})$ turns out to be much easier than 
proving an exponential decay for solutions of $(\ref{Introe1})$ since the solutions of $(\ref{Introe2})$ are explicit. Inequality $(\ref{Introe3})$ could be considered as a 'weak' coercive inequality. The details of this technique will be explained in section 3 (see Lemmas $\ref{Lemma0}$, $\ref{Lemma1}$, $\ref{Lemma3}$, $\ref{Lemma2}$ and $\ref{LemmaAmmari}$). 
\\ Consider the dissipative inequality for $(\ref{DS-Intro})$
$$\partial_t\|f\|_{L^2}^2=-\int_{\mathbb{T}^d\times\mathbb{R}^d}\sigma(x)|\bar{f}-f|^2dxdv,$$
we can see that the damping $\int_{\mathbb{T}^d\times\mathbb{R}^d}\sigma(x)|\bar{f}-f|^2dxdv$ is too strong to lead to a polynomial decay. A reasonable question is if we can get a polynomial decay with a weaker damping. We  give an example where the damping is quite weak 
$$\partial_t\|f\|_{L^2}^2=-\int_{\mathbb{T}^d\times\mathbb{R}^d}|(1-\Delta_x)^{-\epsilon/2}\sigma(x)(\bar{f}-f)|^2dxdv,$$
where $\epsilon$ is a positive constant. Since the order of the pseudo-differential operator $(1-\Delta_x)^{-\epsilon/2}$ is $-\epsilon$, it leads to a polynomial decay and this is the result of Theorem $\ref{TheoremTransportexponentialdecay3}$. 
\\ Another question is that: our method works well for kinetic models of collisionless particles, could it be applied to more sophisticated models? The answer is yes. We also succeed to apply our technique to study the convergence toward equilibrium  for the linearized Boltzmann equation (see Theorem $\ref{TheoremLinearizedBoltzmannDecay}$). In the context of the linearized Boltzmann equation, the main tool to prove the exponential and polynomial convergence toward the equilibrium is based on the spectral gap estimate for the hard potential case and the coercivity estimate for the soft potential case. Using this technique, C. Mouhot has proved exponential decays in the case of hard potential (see  \cite{BarangerMouhot:2005:ESG}, \cite{Mouhot:2006:ECE}, \cite{MouhotNeumann:2006:QPS}, \cite{MouhotStrain:2007:SGC}). For the soft potential case, R. Strain and Y. Guo have proved results about the almost exponential decay (which means that the convergence is faster than any polynomial convergence) in \cite{StrainGuo:2006:AED} , or some exponential decay of the type $\exp(-t^p)$ , $(p<1)$ in \cite{StrainGuo:2008:EDS}. However, obtaining spectral gap and coercivity estimates is sometimes very hard. Using our tools, we can prove an exponential decay for the hard potential  case and an almost exponential decay for the soft potential case. Since we do not need the coercivity of the collision operator, we do not really need assumptions on the collision kernel $B(|v-v_*|,\cos\theta)$ including the smoothness, convexity,... The linearized Boltzmann collision operator is usually split into two parts
$$L[f]=\nu(v)f-Kf,$$
where $\nu(v)f$ is the dominant part. If $K$ is good enough, the spectrum of $L$ is included in the spectrum of $\nu(v)$, which leads to the coercivity of $L$. Our idea is to consider the 'weak' coercivity of $L$ for only a small class of functions: the solutions of $(\ref{Introe2})$. For a solution $g$ of $(\ref{Introe2})$, the integral $\int_0^TL(g)dt$ is equivalent to $T\nu(v)g-C(T)Kg$ in some sense, where $C(T)<<T$. This means that $C(T)Kg$ is absorbed by $T\nu(v)g$ when $T$ is large and we still have the 'weak' coercivity  of $L$ without assuming more conditions on $K$. The only assumption  we need is that the usual dominant part in the linearized Boltzmann collision kernel remains dominant with our very general conditions (see assumptions $(\ref{B1})$, $(\ref{B3})$). These assumptions is the least property that we could expect from the linearized Boltzmann collision operator and they cover both cases: with and without Grad cut-off assumptions. Similar as in the case of the Goldstein-Taylor and related models, our proof remains true if the collision kernel $B(|v-v_*|,\cos\theta)$ depends on the space variable, which means that the effect of the collision of particles depends also on the position where they collide; however, we have not found any real model for this. 
\\ The plan of the paper is the following: the main results of the paper is stated in Section 2 and the main tool of the proofs is studied in Section 3. Sections 3, 4, 5, 6 are devoted to the proofs of Theorems $\ref{TheoremTransportexponentialdecay1}$, $\ref{TheoremTransportexponentialdecay2}$, $\ref{TheoremTransportexponentialdecay3}$ and $\ref{TheoremLinearizedBoltzmannDecay}$.
\section{Preliminaries and Statements of the Main Results}
\subsection{Stabilization of the Goldstein-Taylor equation and related models}
We consider the Goldstein-Taylor model  
\begin{equation}\label{G-T}
\left\{ \begin{array}{ll} \frac{\partial u}{\partial t}+\frac{\partial u}{\partial x}=\sigma(x)(v-u),\vspace{.1in}\\
\frac{\partial v}{\partial t}-\frac{\partial v}{\partial x}=\sigma(x)(u-v),\end{array}\right.
\end{equation}
where $u:=u(t,x)$, $v:=v(t,x)$, $x\in\mathbb{T}=\mathbb{R}/\mathbb{Z}$, $t\geq 0$, with the initial condition 
\begin{equation}\label{G-Tinitial}
u(0,x)=u_0(x), ~~~ v(0,x)=v_0(x).
\end{equation}
Suppose that $\sigma\in L^2(\mathbb{T})$. Define the asymptotic profile of the system $(\ref{G-T})$:
\begin{equation}\label{AsymProfileu}
(u_\infty,v_\infty)=\left(\frac{1}{2}\int_{\mathbb{T}}(u_0+v_0)dx,\frac{1}{2}\int_{\mathbb{T}}(u_0+v_0)dx\right),
\end{equation}
and the energy is then
\begin{equation}\label{Energyu}
H_u(t)=\int_{\mathbb{T}}[(u-u_\infty)^2+(v-v_\infty)^2]dx.
\end{equation}
\\ We also consider the following non-homogeneous (in space) transport equation
\begin{equation}
\label{DS-transport}
\frac{\partial f}{\partial t}+v.\nabla f=\sigma(x)(\bar{f}-f),
\end{equation}
where $f:=f(t,x,v)$ is the density of particles at time $t$, position $x$ and velocity $v$. The notation $\bar{f}$ is $\int_{V}f(t,x,v)$, where $V$ is a bounded set of $\mathbb{R}^d$ of measure $1$. The solutions are considered of periodic $1$ or on $\mathbb{T}^d=\mathbb{R}^d/\mathbb{Z}^d$. 
We give an example where the damping is week enough to give a polynomial decay
\begin{equation}
\label{DS-transportspecial}
\frac{\partial f}{\partial t}+v.\nabla f=\sigma(x)(1-\Delta_x)^{-\frac{\epsilon}{2}}\sigma(x)(\bar{f}-f),
\end{equation}
where $\epsilon$ is a positive constant. The initial data is
\begin{equation}
\label{DS-initial}
f(0,x,v)=f_0(x,v).
\end{equation}
Define the energy of $(\ref{DS-transportspecial})$
\begin{equation}
\label{Energyf}
E_f(t)=\int_{\mathbb{T}^d}\int_V|f-f_\infty|^2dvdx,
\end{equation}
where 
\begin{equation}
\label{AsymProfilef}
f_\infty=\int_{\mathbb{T}^d}\int_Vf_0(x,v)dxdv.
\end{equation}
Our main results are
\begin{theorem}\label{TheoremTransportexponentialdecay1}
 When $\sigma\geq0$, $\sigma\in L^2(\mathbb{T}^d)$, $\sigma\ne0$, $f_0\in L^2(\mathbb{T}^d\times\mathbb{R}^d)$, the solution of the equation $(\ref{G-T})$ decays exponentially in time towards the equilibrium state of the equation. 
\end{theorem}
\begin{theorem}\label{TheoremTransportexponentialdecay2}
 When $\sigma\geq0$, $\sigma\in H^1(\mathbb{T}^d)$, $\sigma\ne0$, $f_0\in L^2(\mathbb{T}^d\times\mathbb{R}^d)\cap L^\infty(\mathbb{T}^d\times\mathbb{R}^d)$, the solution of the equation $(\ref{DS-transport})$ decays exponentially in time towards the equilibrium state of the equation. 
\end{theorem}
\begin{remark} Compare to the results in \cite{DesvillettesSalvarani:2009:ABD}, our results not only improve the type of convergence from polynomial to exponential, but also relax the regularity on the initial condition and the cross section. Moreover, we do not need the condition that the cross section $\sigma$ is greater than a positive polynomial.  
\end{remark}
\begin{theorem}\label{TheoremTransportexponentialdecay3}
 When $\sigma\geq0$, $\sigma\in C^\infty(\mathbb{T}^d)$, $\sigma\ne0$,  $f_0\in C^\infty(\mathbb{T}^d\times\mathbb{R}^d)$, the solution of the equation $(\ref{DS-transportspecial})$ decays polymonially in the following sense $\forall M>0$, there exist positive constants $C(M)$ and $k>M$ such that
 \begin{equation}
 \label{mainresult}
H_f(t)\leq {C(M)}(t+1)^{-k}\|f_0-f_\infty\|^2_{H^{\epsilon}}.
\end{equation}
\end{theorem}
\begin{remark} The existence of a solution of this equation can be proved by a Picard iteration technique; however, we do not go into details of this classical proof.
\end{remark}
\begin{remark}
Since the order of the pseudo-differential operator $(1-\Delta_x)^{-\epsilon/2}$ is $-\epsilon$  in $(\ref{DS-transportspecial})$, which means that  the damping is quite weak, we get a polynomial decay. According to our theorem the order of the convergence is ${-\infty}$, or we can get an almost exponential decay with this damping.
\end{remark}
\subsection{Stabilization of the linearized Boltzmann equation}
The Boltzmann equation describes the behavior of a dilute gas when the interactions are binary (see \cite{CercignaniIllnerPulvirenti:1994:MTD}, \cite{Glassey:1996:CPK}, \cite{Villani:2002:RMT})
\begin{equation}\label{Boltzmann}
\partial_t F+v.\nabla_x F=Q(F,F), t\geq 0, x\in \mathbb{T}^d, v\in\mathbb{R}^d.
\end{equation}
In $(\ref{Boltzmann})$, $Q$ is the quadratic Boltzmann collision operator, defined by
$$Q(F,F)=\int_{S^{N-1}}\int_{\mathbb{R}^N}(F'F'_*-FF_*)B(|v-v_*|,\cos\theta)d\sigma dv_*,$$
where $F=F(t,x,v)$, $F_*=F(t,x,v_*)$, $F'_*(t,x,v'_*)$, $F'=F(t,x,v')$ in which
$$v'=\frac{v+v_*}{2}+\frac{|v-v_*|}{2}\sigma;v'_*=\frac{v+v_*}{2}-\frac{|v-v_*|}{2}\sigma,\sigma\in\mathbb{S}^{N-1}.$$
This is the so called "$\sigma$-representation" of the Boltzmann collision operator. Up to a Jacobian factor $2^{N-2}\sin^{N-2}(\theta/2)$, where $\cos\theta=(v'_*-v').(v_*-v)/|v_*-v|^2$, one can also define the alternative "$\omega$-representation",
$$Q(F,F)=\int_{S^{N-1}}\int_{\mathbb{R}^N}(F'F'_*-FF_*)\mathcal{B}(v-v_*,\omega) dv_*d\omega,$$
with 
$$v'=v+((v_*-v).\omega)\omega,v'_*=v_*-((v_*-v).\omega)\omega,\omega\in\mathbb{S}^{d-1},$$
and
$$\mathcal{B}(v-v_*,\omega)=2^{N-2}\sin^{N-2}(\theta/2)B(|v-v_*|,\cos\theta).$$
\\ The equilibrium distribution is given by the Maxwellian distribution
\begin{equation}\label{BoltzmannMaxelllian}
M(\rho,u,T)(v)=\frac{\rho}{(2\pi T)^{\frac{N}{2}}}\exp\left(-\frac{|u-v|^2}{2T}\right),
\end{equation}
where $\rho$, $u$, $T$ are the density, mean velocity and temperature of the gas at the point $x$
\begin{equation}
\rho=\int_{\mathbb{R}^d}f(v)dv,u=\frac{1}{\rho}\int_{\mathbb{R}^d}vf(v)dv,T=\frac{1}{N\rho}\int_{\mathbb{R}^d}|u-v|^2f(v)dv.
\end{equation}
Denote by 
$$\mu(v)=(2\pi)^{-d/2}\exp(-|v|^2/2),$$ 
the normalized unique equilibrium with mass $1$, momentum $0$ and temperature $1$, we consider $F$ to be a solution of the equation near $\mu$. Put $F=\mu+\sqrt{\mu}f$, then
\begin{equation}\label{BoltzmannNearEquilibrium}
\partial_t f+v.\nabla_x f=2\mu^{-1/2}Q(\mu,\sqrt{\mu}f)+\mu^{-1/2}Q(\sqrt{\mu}f,\sqrt{\mu}f).
\end{equation}
Define $$\Gamma(f,f)=\mu^{-1/2}Q(\sqrt{\mu}f,\sqrt{\mu}f),$$
and $$L[f]=2\mu^{-1/2}Q(\mu,\sqrt{\mu}f),$$
 the following equation is the linearized Boltzmann equation
\begin{equation}\label{BoltzmannLinearized}
\partial_t f+v.\nabla_x f = L[f],
\end{equation}
where $L[f]=$
\begin{equation}\label{BoltzmannLinearizedCollision}
\footnotesize{\int_{\mathbb{R}^N\times\mathbb{S}^{N-1}}2\mathcal{B}\mu^{1/2}(v)\mu(v_*)[\mu^{1/2}(v')f(v'_*)+\mu^{1/2}(v_*')f(v')-\mu^{1/2}(v_*)f(v)-\mu^{1/2}(v)f(v_*)]dv_*d\sigma.}
\end{equation}
We assume the following conditions on the collision kernel $\mathcal{B}$
\\ $(\mathbb{B}_1)$ There exist a constant $\alpha>-d+1$ and a positive constant $M_1$ such that 
\begin{equation}\label{B1}
\int_{\mathbb{R}^d\times\mathbb{S}^{d-1}}\mu(v_*)\mathcal{B}(|v-v_*|,\omega)d\omega dv_*\geq M_1(|v|+1)^\alpha.
\end{equation}
$(\mathbb{B}_2)$ There exist constants $1-d<\beta\leq\alpha+2/3$, and $M_2>0$ such that 
\begin{equation}\label{B3}
\mathcal{B}(|v-v_*|,\omega)\leq M_2|v-v_*|^\beta|v'-v|^{d-2}.
\end{equation}
We impose these conditions to assure that the term
\begin{eqnarray*}
\int_{\mathbb{R}^d\times\mathbb{S}^{d-1}}\mathcal{B}(|v-v_*|,\omega)\mu(v_*)f(v)dv_*d\sigma
\end{eqnarray*}
is the dominant term in the linearized Boltzmann collision operator. These assumptions cover both cases: with and without Grad cut-off.
\\ Consider the energy of $f$
\begin{equation}\label{BoltzmannEnergy}
H_f(t)=\int_{\mathbb{T}^d\times\mathbb{R}^d}|f|^2dxdv,
\end{equation}
and its derivative in time
\begin{eqnarray}
& &\frac{d}{dt}H_f(t)\\\nonumber
&=&-\frac{1}{2}\int_{\mathbb{T}^d\times\mathbb{R}^d\times\mathbb{R}^d\times\mathbb{S}^{d-1}}\mathcal{B}\mu_*\mu\times\\\nonumber
& &\times[f_*'\mu_*'^{-1/2}+f'\mu'^{-1/2}-f_*\mu^{-1/2}_*-f\mu^{-1/2}]^2d\sigma dv_*dvdx\\\nonumber
&\leq&0,
\end{eqnarray}
where we use the notation $f'_*=f(v'_*)$, $f'=f(v')$, $f_*=f(v_*)$, $f=f(v)$, $\mu'_*=\mu(v'_*)$, $\mu'=\mu(v')$, $\mu_*=\mu(v_*)$ and $\mu=\mu(v)$.
\\ For $\rho\in\mathbb{R}$, define
 $$L^2((|v|+1)^\rho):=\{f | (|v|+1)^\rho f\in L^2(\mathbb{T}^d\times\mathbb{R}^d)\}.$$ 
Denote by $\mathcal{S}(t)f_0$ the solution of the linearized Boltzmann equation and suppose that $f_0$ is orthogonal to the kernel of the linearized Boltzmann collision kernel: 
$$\int_{\mathbb{R}^d}\mu^{1/2}f_0dv=\int_{\mathbb{R}^d}\mu^{1/2}|v_i|f_0dv=\int_{\mathbb{R}^d}\mu^{1/2}|v|^2f_0dv=0,$$
for all $i\in\{1,\dots,d\}$. 
\begin{theorem}\label{TheoremLinearizedBoltzmannDecay} With the assumptions $(\mathbb{B}_1)$ and $(\mathbb{B}_2)$:
\begin{itemize}
\item The 'hard potential' case $\alpha,\beta> 0$: suppose that $f_0\in L^2(\mathbb{T}^d\times\mathbb{R}^d)$, there exist positive constants $M_0$, $\delta$ such that
\begin{equation}\label{LinearizedBoltzmannExDecay}
\left\|\mathcal{S}(t)\left(f_0-\int_{\mathbb{R}^d}f_0dv\right)\right\|_{L^2}\leq M_0\exp(-\delta t)\left\|f_0-\int_{\mathbb{R}^d}f_0dv\right\|_{L^2}.
\end{equation}
\item The 'soft potential' case $-(d-1)<\alpha,\beta<0$, : suppose that $f_0\in L^2((|v|+1)^\delta)$, $(\delta>0)$, for any $M_1>0$, there exist $p>M_1$ and $M_2>0$ such that
\begin{equation}\label{LinearizedBoltzmannPolyDecay}
\left\|\mathcal{S}(t)\left(f_0-\int_{\mathbb{R}^d}f_0dv\right)\right\|_{L^2}\leq M_2t^{-p}\left\|f_0-\int_{\mathbb{R}^d}f_0dv\right\|_{L^2}.
\end{equation}
\end{itemize} 
\end{theorem}
\begin{remark} In this theorem, since we prove a 'weak' coercive estimate instead of  spectral gap and coercivity estimates for the linearized Boltzmann operator, we can get exponential and almost exponential decays without requiring too much assumptions on the collision kernel including the smoothness, convexity, ...  The only property that we need is that the dominant term remains dominant with our conditions $(\mathbb{B}_1)$ and $(\mathbb{B}_2)$. 
\end{remark}
\begin{remark}
Our proof works well also for the case where $B$ depends on x; however, we have not found any real application for this. 
\end{remark}
\section{The main tool}
Let $(H,<.,.>,\|.\|)$ be a real Hilbert space with its inner product and its norm, $\mathcal{A}$ be an operator on $H$ satisfying $<\mathcal{A}(x),x>=0$ for all $x$ in $H$ and $\mathcal{K}$ be a self-ajoint linear operator. Suppose that
$$<\mathcal{K}(x),y>=<x,\mathcal{K}(y)>=<\mathcal{K}^{1/2}(x),\mathcal{K}^{1/2}(y)>.$$
Let $f$ be the solution of the evolution equation
\begin{equation}\label{EquationAB}
\left\{ \begin{array}{ll} \frac{\partial f}{\partial t}+\mathcal{A}(f)=-\mathcal{K}(f), t\in\mathbb{R}_+,\vspace{.1in}\\
f(0)=f_0, f_0\in H,\end{array}\right.
\end{equation}
and let $g$ be the solution of 
\begin{equation}\label{EquationA}
\left\{ \begin{array}{ll} \frac{\partial g}{\partial t}+\mathcal{A}(g)=0, t\in\mathbb{R}_+,\vspace{.1in}\\
g(0)=f_0.\end{array}\right.
\end{equation}
\begin{lemma}\label{Lemma0}
 For all $T$ in $\mathbb{R}_+$
\begin{equation}\label{fg0}
\int_0^T\|\mathcal{K}^{1/2}(f)\|^2dt\leq\int_0^T\|\mathcal{K}^{1/2}(g)\|^2dt.
\end{equation}
\end{lemma}
\begin{proof}
Consider  the norm of $\|f-g\|^2$
\begin{eqnarray*}
\|f-g\|^2(T) & = & 2\int_0^T<\partial_t f-\partial_t g,f-g>dt\\
& = & -\int_0^T 2<\mathcal{K}^{1/2}(f),\mathcal{K}^{1/2}(f-g)>dt\\
&=&-2\int_0^T\| \mathcal{K}^{1/2}(f)\|^2dt+2\int_0^T<\mathcal{K}^{1/2}(f),\mathcal{K}^{1/2}(g)>dt\\
&\leq &-\int_0^T\|\mathcal{K}^{1/2}(f)\|^2dt+\int_0^T\|\mathcal{K}^{1/2}(g)\|^2dt,
\end{eqnarray*}
which leads to
\begin{eqnarray*}
\int_0^T\|\mathcal{K}^{1/2}(f)\|^2dt&\leq &\int_0^T\|\mathcal{K}^{1/2}(g)\|^2dt.
\end{eqnarray*}
\end{proof}
\begin{lemma}\label{Lemma1} 
If $\mathcal{K}^{1/2}$ is bounded, then for all $T$ in $\mathbb{R}_+$
\begin{equation}\label{fg}
M_1\int_0^T\|\mathcal{K}^{1/2}(g)\|^2dt\leq \int_0^T\|\mathcal{K}^{1/2}(f)\|^2dt,
\end{equation}
where $M_1$ is a positive constant.
\end{lemma}
\begin{proof}
Take the derivative in time of $\|f-g\|^2$
\begin{eqnarray*}
\partial_t\|f-g\|^2 & = & 2<\partial_t f-\partial_t g,f-g>\\
& = & -2<\mathcal{K}^{1/2}(f),\mathcal{K}^{1/2}(f-g)>\\
& \leq & \| \mathcal{K}^{1/2}(f)\|^2+\| \mathcal{K}^{1/2}(f-g)\|^2\\
& \leq & \| \mathcal{K}^{1/2}(f)\|^2+C\|f-g\|^2,
\end{eqnarray*}
the last inequality follows from the boundedness of $\mathcal{K}^{1/2}(f-g)$, where $C$ is a positive constant. Gronwall's inequality then leads to 
\begin{eqnarray*}
\|f-g\|^2(t) \leq\int_0^t\exp(C(t-s))\| \mathcal{K}^{1/2}(f)\|^2ds,
\end{eqnarray*}
which together with the boundedness of $\mathcal{K}^{1/2}(f-g)$ leads to 
\begin{eqnarray*}
\| \mathcal{K}^{1/2}(f-g)\|^2(t) \leq C\exp(Ct)\int_0^t \| \mathcal{K}^{1/2}(f)\|^2ds,
\end{eqnarray*}
where $C$ is some positive constant.
This deduces
\begin{eqnarray*}
\int_0^T \| \mathcal{K}^{1/2}(f-g)\|^2dt \leq CT\exp(CT)\int_0^T \| \mathcal{K}^{1/2}(f)\|^2dt.
\end{eqnarray*}
The triangle inequality deduces
\begin{eqnarray*}
\int_0^T \| \mathcal{K}^{1/2}(g)\|^2 dt \leq (CT\exp(CT)+1)\int_0^T \| \mathcal{K}^{1/2}(f)\|^2dt.
\end{eqnarray*}

\end{proof}
\begin{lemma}\label{Lemma3} Let $(H',\|.\|_0)$ be a Banach subspace of $H$ with its norm. Suppose that for any $h$ in $H'$, $\|h\|\leq M\|h\|_0$, where $M$ is a positive constant and that for any solution $g$ of $(\ref{EquationA})$
\begin{equation}
\label{Lemma3D0}\|f_0\|=\|g(t)\|,\forall t\in\mathbb{R}_+.
\end{equation} 
We assume that for any positive constant $\epsilon$, the operator $\mathcal{K}$ could be decomposed into the sum of two linear operators $\mathcal{K}_{\epsilon,1}$ and $\mathcal{K}_{\epsilon,2}$ such that 
\begin{eqnarray}
\label{Lemma3D1} & & \mathcal{K}=\mathcal{K}_{\epsilon,1}+\mathcal{K}_{\epsilon,2},\\
\label{Lemma3D2} & &\| \mathcal{K}^{1/2}\|^2=\| {\mathcal{K}_{\epsilon,1}}^{1/2}\|^2+\| {\mathcal{K}_{\epsilon,2}}^{1/2}\|^2,\\
\label{Lemma3D3} & &\| {\mathcal{K}_{\epsilon,1}}^{1/2}(h)\|\leq C_1(\epsilon)\|h\|, ~~\forall h\in H',\\
\label{Lemma3D4} & &\| {\mathcal{K}_{\epsilon,2}}^{1/2}(h)\|\leq C_2(\epsilon)\|h\|_0, ~~\forall h\in H',\\
\label{Lemma3D4post} & &\| {\mathcal{K}}^{1/2}(h)\|\leq C(\mathcal{K})\|h\|_0, ~~\forall h\in H',
\end{eqnarray}
where $C_1(\epsilon)$, $C_2(\epsilon)$ and $C(\mathcal{K})$ are positive constants, $C_2(\epsilon)$ tends to $0$ as $\epsilon$ tends to $0$, and $\mathcal{K}^{1/2}_{\epsilon,i},$ $i\in\{1,2\}$ are defined in the following way
$$<\mathcal{K}_{\epsilon,i}(h),k>=<\mathcal{K}_{\epsilon,i}^{1/2}(h),\mathcal{K}_{\epsilon,i}^{1/2}(k)>, ~~~\forall h,k \in H'.$$
Suppose that there exist positive numbers $T_0$ and $C$ such that
\begin{equation}\label{Lemma3D5pre}
\int_0^{T_0}\|\mathcal{K}^{1/2}(g)\|^2dt\geq C\|f_0\|_0^2.\end{equation}
Then there exists a constant $M_1$ depending on $T_0$ such that
\begin{equation}\label{Lemma3D5}
M_1\int_0^T\|\mathcal{K}^{1/2}(g)\|^2dt\leq \int_0^T\|\mathcal{K}^{1/2}(f)\|^2dt,
\end{equation}
for all $T\geq T_0$.
\end{lemma}
\begin{proof} Similar as in the previous lemma
\begin{eqnarray*}
\partial_t\|f-g\|^2 & = & -2<\mathcal{K}^{1/2}_{\epsilon,1}(f),\mathcal{K}_{\epsilon,1}^{1/2}(f-g)>-2<\mathcal{K}^{1/2}_{\epsilon,2}(f),\mathcal{K}_{\epsilon,2}^{1/2}(f-g)>\\
& \leq & \| \mathcal{K}^{1/2}_{\epsilon,1}(f)\|^2+ \|\mathcal{K}_{\epsilon,1}^{1/2}(f-g)\|^2-2<\mathcal{K}^{1/2}_{\epsilon,2}(f),\mathcal{K}_{\epsilon,2}^{1/2}(f-g)>  \\
& \leq & \| \mathcal{K}^{1/2}_{\epsilon,1}(f)\|^2+ C_1(\epsilon)^2\|f-g\|^2+2\|\mathcal{K}^{1/2}_{\epsilon,2}(f)\|\|\mathcal{K}_{\epsilon,2}^{1/2}(f-g)\|, \end{eqnarray*}
the last inequality follows from $(\ref{Lemma3D3})$. Gronwall's inequality deduces
\begin{eqnarray*}
\|f-g\|^2(t)& \leq &\int_0^t(\| \mathcal{K}^{1/2}_{\epsilon,1}(f)\|^2+2\|\mathcal{K}^{1/2}_{\epsilon,2}(f)\|\|\mathcal{K}_{\epsilon,2}^{1/2}(f-g)\|)\exp(C_1(\epsilon)^2(t-s))ds \\
& \leq &\exp(C_1(\epsilon)^2t)\int_0^t(\| \mathcal{K}^{1/2}_{\epsilon,1}(f)\|^2+2\|\mathcal{K}^{1/2}_{\epsilon,2}(f)\|\|\mathcal{K}_{\epsilon,2}^{1/2}(f-g)\|)ds. 
\end{eqnarray*}
The previous inequality implies
\begin{equation}\label{Lemma3D6}
\int_0^T\|f-g\|^2dt \leq T\exp(C_1(\epsilon)^2T)\int_0^T(\| \mathcal{K}^{1/2}_{\epsilon,1}(f)\|^2+2\|\mathcal{K}^{1/2}_{\epsilon,2}(f)\|\|\mathcal{K}_{\epsilon,2}^{1/2}(f-g)\|)dt, 
\end{equation}
for any $T>T_0$.
The two inequalities $(\ref{Lemma3D3})$ and $(\ref{Lemma3D6})$ lead to
\begin{eqnarray*}
& &\int_0^T\| \mathcal{K}^{1/2}_{\epsilon,1}(f-g)\|^2dt \\\nonumber
& \leq &TC_1(\epsilon)^2\exp(C_1(\epsilon)^2T)\int_0^T(\| \mathcal{K}^{1/2}_{\epsilon,1}(f)\|^2+2\|\mathcal{K}^{1/2}_{\epsilon,2}(f)\|\|\mathcal{K}_{\epsilon,2}^{1/2}(f-g)\|)dt. 
\end{eqnarray*} 
Apply the triangle inequality to the previous inequality to get
\begin{eqnarray}\label{Lemma3D7}
& &\int_0^T\| \mathcal{K}^{1/2}_{\epsilon,1}(g)\|^2dt \\\nonumber
& \leq & 2(TC_1(\epsilon)^2\exp(C_1(\epsilon)^2T)+1)\int_0^T(\| \mathcal{K}^{1/2}_{\epsilon,1}(f)\|^2+2\|\mathcal{K}^{1/2}_{\epsilon,2}(f)\|\|\mathcal{K}_{\epsilon,2}^{1/2}(f-g)\|)dt. 
\end{eqnarray}
The there inequalities $(\ref{Lemma3D4})$, $(\ref{Lemma3D4post})$ and $(\ref{Lemma3D5pre})$ imply that for $\epsilon$ small enough
\begin{equation}\label{Lemma3D8}
\int_0^T\| \mathcal{K}^{1/2}_{\epsilon,1}(g)\|^2dt\geq C(\epsilon) \int_0^T\| \mathcal{K}^{1/2}(g)\|^2dt,
\end{equation}
where $C(\epsilon)$ is a positive constant depending on $\epsilon$. 
\\ Combine $(\ref{Lemma3D7})$ and $(\ref{Lemma3D8})$ to get
\begin{eqnarray}\label{Lemma3D9a}
& & \frac{C(\epsilon)}{2(TC_1(\epsilon)^2\exp(C_1(\epsilon)^2T)+1) } \int_0^T\| \mathcal{K}^{1/2}(g)\|^2dt\\\nonumber
& \leq & \int_0^T(\| \mathcal{K}^{1/2}_{\epsilon,1}(f)\|^2+2\|\mathcal{K}^{1/2}_{\epsilon,2}(f)\|\|\mathcal{K}_{\epsilon,2}^{1/2}(f-g)\|)dt. 
\end{eqnarray}
Since for any positive constant $\delta$
\begin{eqnarray*}
& & \int_0^T(\| \mathcal{K}^{1/2}_{\epsilon,1}(f)\|^2+2\|\mathcal{K}^{1/2}_{\epsilon,2}(f)\|\|\mathcal{K}_{\epsilon,2}^{1/2}(f-g)\|)dt\\
&
\leq &\int_0^T\left(\| \mathcal{K}_{\epsilon,1}^{1/2}(f)\|^2+\frac{1}{\delta}\|\mathcal{K}^{1/2}_{\epsilon,2}(f)\|^2+\delta\|\mathcal{K}_{\epsilon,2}^{1/2}(f-g)\|^2\right)dt\\
&\leq &\left(1+2\delta+\frac{1}{\delta}\right)\int_0^T\| \mathcal{K}^{1/2}(f)\|^2dt+2\delta\int_0^T\|\mathcal{K}^{1/2}(g)\|^2dt,
\end{eqnarray*}
Inequality $(\ref{Lemma3D9a})$ leads to
\begin{eqnarray*}
& &\left(\frac{C(\epsilon)}{2(TC_1(\epsilon)^2\exp(C_1(\epsilon)^2T)+1)}-2\delta \right)\int_0^T\| \mathcal{K}^{1/2}(g)\|^2dt\\
&\leq&\left(1+2\delta+\frac{1}{\delta}\right)\int_0^T\| \mathcal{K}^{1/2}(f)\|^2dt, 
\end{eqnarray*}
which implies $(\ref{Lemma3D5pre})$ for $\delta$ small enough.
\end{proof}
\begin{remark} Lemma $\ref{Lemma1}$ will be used later for the case of Goldstein-Taylor and related models, while Lemma $\ref{Lemma3}$ will be used for the linearized Boltzmann equation.
\end{remark}
\begin{lemma}\label{Lemma2}
Suppose that $\mathcal{K}$ satisfies the conditions in Lemmas $\ref{Lemma1}$ or $\ref{Lemma3}$ and that there exist positive numbers $T_0$ and $C$ such that
\begin{equation}\label{Observe}
\int_0^{T_0}\|\mathcal{K}^{1/2}(g)\|^2dt\geq C\|f_0\|^2,
\end{equation}
then there exist positive numbers $T_1$ and $\delta$ such that for all $t\geq T_1$
\begin{equation}\label{ExDecay}
\|f(t)\|\leq \exp(-\delta t)\|f_0\|.
\end{equation}
Moreover, $(\ref{ExDecay})$ also leads to $(\ref{Observe})$.
\end{lemma}
\begin{proof}
{\\\bf Step 1:} $(\ref{Observe})$ leads to $(\ref{ExDecay})$.
\\ Choose $T=kT_0$, where $k$ is a positive integer. Since 
$$\|f(0)\|-\|f(T)\|=\int_0^T\|\mathcal{K}^{1/2}(f)\|^2 dt,$$
there exists $p$ in $\{0,\dots,k-1\}$ such that
$$\frac{\|f(0)\|}{k}\geq \int_{pT_0}^{(p+1)T_0}\|\mathcal{K}^{1/2}(f)\|^2 dt.$$
Let $h$ be the solution of
$$\partial_th+A(h)=0,$$
with $h(0)=f(pT_0).$ Inequality $(\ref{Observe})$ implies that
$$\int_{0}^{T_0}\|\mathcal{K}^{1/2}(h)\|^2dt\geq C\|f(pT_0)\|,$$
which together with Lemma $\ref{Lemma0}$ deduces 
$$\int_{0}^{T_0}\|\mathcal{K}^{1/2}(f)\|^2dt\geq C\|f(pT_0)\|.$$
This leads to
$$\|f(kT_0)\|\leq\|f(pT_0)\|\leq \frac{1}{Ck}\|f(0)\|,$$
where $C$ is some positive constant,since 
$$\partial_t\|f\|^2=-2<\mathcal{K}^{1/2}f, \mathcal{K}^{1/2}f>,$$
or $\|f\|$ is decreasing; for $k$ large enough. The previous inequality implies 
$$\|f(T_*)\|\leq \exp(-\delta_*T_*)\|f(0)\|,$$
where $T_*=kT_0$ and $\delta_*=\frac{\ln({Ck})}{T_*}$. Then for $t\in[T_*,2T_*)$,
$$\|f(t)\|\leq \|f(T_*)\|\leq\exp(-\frac{\delta_*}{2}2T_*)\|f(0)\|\leq \exp(-\frac{\delta_*}{2}t)\|f(0)\|.$$
This leads to 
$$\|f(t)\|\leq \exp(-\frac{\delta_*}{2}(t-T_*))\|f(T_*)\|\leq \exp(-\frac{\delta_*}{2}t)\|f(0)\|,$$
for $t\in[2T_*,3T_*)$. An induction argument leads to the exponential decay $(\ref{ExDecay})$ with $\delta=\frac{\delta_*}{2}$.
{\\\bf Step 2:} $(\ref{ExDecay})$ leads to $(\ref{Observe})$.
\\ Inequality $(\ref{ExDecay})$ deduces that there exist constants $C<1$ and $T_*>0$ such that for $T>T_*$
$$\|f(0\|-\|f(T)\|=\int_0^T\|\mathcal{K}^{1/2}(f)\|^2dt\geq C \|f(0)\|^2.$$ Lemmas $\ref{Lemma1}$ and $\ref{Lemma3}$ imply that
$$\int_0^T\|\mathcal{K}^{1/2}(g)\|^2dt\geq C \|f(0)\|^2.$$
\end{proof}
We also recall Lemma 4.4 in \cite{AmmariTucsnak:2001:SSO}, for a proof of this lemma we refer to \cite{AmmariTucsnak:2000:SBE} and \cite{JaffardTucsnakZuazua:1998:SIS}.
\begin{lemma}\label{LemmaAmmari}
Let $\{\mathcal{E}_k\}$ be a sequence of positive real numbers satisfying 
$$\mathcal{E}_{k+1}\leq\mathcal{E}_k-C\mathcal{E}^{2+\zeta}_{k+1},\forall k\geq0,$$
where $C>0$ and $\zeta>-1$ are constants. Then there exists a positive constant $M$, such that
$$\mathcal{E}_k\leq \frac{M}{(k+1)^{\frac{1}{1+\zeta}}},k\geq0.$$
\end{lemma}
\section{Decay rates of the Goldstein-Taylor model}
Consider the following system:
\begin{equation}\label{TransportSystem}
\left\{ \begin{array}{ll} \frac{\partial\varphi}{\partial t}+\frac{\partial\varphi}{\partial x}=0,\vspace{.1in}\\
\frac{\partial\phi}{\partial t}-\frac{\partial\phi}{\partial x}=0,\end{array}\right.
\end{equation}
where $\varphi:=\varphi(t,x)$, $\phi:=\phi(t,x)$, $x\in\mathbb{T}=\mathbb{R}/\mathbb{Z}$, $t\geq 0$, with the initial condition 
\begin{equation}\label{TransportSysteminitial}
\varphi(0,x)=\varphi_0(x), ~~~ \phi(0,x)=\phi_0(x).
\end{equation}
Then asymptotic profile and the energy  of the system are  then
\begin{equation}\label{AsymProfilephi}
(\varphi_\infty,\phi_\infty)=\left(\frac{1}{2}\int_{\mathbb{T}}(\varphi_0+\phi_0)dx,\frac{1}{2}\int_{\mathbb{T}}(\varphi_0+\phi_0)dx\right),
\end{equation}
and 
\begin{equation}\label{Energyphi}
H_\varphi(t)=\int_{\mathbb{T}}[(\varphi-\varphi_\infty)^2+(\phi-\phi_\infty)^2]dx.
\end{equation}
The following proposition is a consequence of Lemmas $\ref{Lemma0}$, $\ref{Lemma1}$ and $\ref{Lemma2}$.
\begin{proposition}\label{TheoremExponentialDecays}
Suppose that there exist positive numbers $T_0$ and $\delta$ such that
\begin{equation}\label{ExpoDecay}
\forall t\geq T_0, \forall u_0, v_0 \in W^{1,1}(\mathbb{T}): ~~~~ H_u(t)\leq \exp(-\delta t) H_u(0),
\end{equation}
then there exist a positive number $T_1$ and a nonnegative number $C$ such that
\begin{equation}\label{observability}
\int_0^{T_1}\int_{\mathbb{T}}\sigma (\varphi-\phi)^2 dxdt\geq C\int_{\mathbb{T}}[(\varphi-\varphi_\infty)^2+(\phi-\phi_\infty)^2]dx,
\end{equation} 
for $\varphi_0=u_0$ and $\phi_0=v_0$.
\\Moreover, if there exist $T_1$ and $C$ such that $(\ref{observability})$ satisfies, then there exist $T_0$ and $\delta$ such that $(\ref{ExpoDecay})$ is true.
\end{proposition}
Theorem $\ref{TheoremTransportexponentialdecay1}$ is a direct consequence of Proposition $\ref{TheoremExponentialDecays}$ and the following proposition.
\begin{proposition}\label{ProproEnergy-DampingwithweightH1} There exists a positive constant $T_0$ such that for $T>T_0$
\begin{equation}\label{ObservabilityWithWeightH1}
\int_0^T\int_{\mathbb{T}}\sigma(\varphi-\phi)^2dxdt \geq C(T)\int_{\mathbb{T}}[(\varphi_0-\varphi_\infty)^2+(\phi_0-\phi_\infty)^2]dx.
\end{equation}
\end{proposition}
\begin{proof}
Since $(\varphi,\phi)$ is the solution of the system $(\ref{TransportSystem})$, 
$$(\varphi,\phi)=(\varphi_0(x- t),\phi_0(x+ t)).$$
 Write $\varphi_0$ and $\phi_0$ under the form of Fourier series:
\begin{eqnarray*}
\varphi_0(x)&=& \sum_{-\infty}^\infty\exp(in\pi x)a_n,\\
\phi_0(x)&=& \sum_{-\infty}^\infty\exp(in\pi x)b_n,
\end{eqnarray*}
then
\begin{eqnarray*}
\varphi_0(x-t)&=& \sum_{-\infty}^\infty\exp(in\pi (x-t))a_n,\\
\phi_0(x+t)&=& \sum_{-\infty}^\infty\exp(in\pi (x+t))b_n.
\end{eqnarray*}
Choose $T$ to be a positive integer, the previous formulas imply
\begin{eqnarray*}
& &\int_0^T\sigma(\varphi-\phi)^2dt\\
 &=& \int_0^T\sigma\left|\sum_{-\infty}^\infty\exp(in\pi x)(a_n\exp(-in\pi t)-b_n\exp(in\pi t))\right|^2dt\\
 &=&\lim_{M\to\infty}\left(\sum_{|n|<M}\int_0^T\sigma|a_n\exp(-in\pi t)-b_n\exp(in\pi t)|^2dt+\right.\\
 & &+\sum_{|n|,|m|<M,n\ne m}\exp(i(n-m)\pi x)\times\\
 & &\left.\times\int_0^T[a_n\exp(-in\pi t)-b_n\exp(in\pi t)]\overline{a_m\exp(-im\pi t)-b_m\exp(im\pi t)}dt\right)\\
  &=&\lim_{M\to\infty}\left(\sum_{|n|<M}\int_0^T\sigma|a_n\exp(-in\pi t)-b_n\exp(in\pi t)|^2dt+\right.\\
 & &+\sum_{|n|,|m|<M,n\ne m}\exp(i(n-m)\pi x)\times\\
 & &\left.\times\int_0^T[a_n\exp(-in\pi t)-b_n\exp(in\pi t)]\overline{a_m\exp(-im\pi t)-b_m\exp(im\pi t)}dt\right)\\
 &=&\lim_{M\to\infty}\sum_{|n|<M}\int_0^T\sigma|a_n\exp(-in\pi t)-b_n\exp(in\pi t)|^2dt\\
  &=&\sum_{n\in\mathbb{R},n\ne 0}T\sigma(|a_n|^2+|b_n|^2)+T\sigma|a_0-b_0|^2,
\end{eqnarray*}
which leads to
\begin{equation}\label{GTProof1}
\int_{\mathbb{T}}\int_0^T\sigma(\varphi-\phi)^2dtdx=T\int_{\mathbb{T}}\sigma dx\left(\sum_{n\in\mathbb{R},n\ne 0}(|a_n|^2+|b_n|^2)+|a_0-b_0|^2\right).
\end{equation}
Moreover, the right hand side of $(\ref{ObservabilityWithWeightH1})$ is equal to
\begin{equation}\label{GTProof2}
\int_{\mathbb{T}}[(\varphi_0-\varphi_\infty)^2+(\phi_0-\phi_\infty)^2]dx=\sum_{n\in\mathbb{R},n\ne 0}(|a_n|^2+|b_n|^2).
\end{equation}
Inequality $(\ref{ObservabilityWithWeightH1})$ follows by $(\ref{GTProof1})$ and $(\ref{GTProof2})$.
\end{proof}
\section{Decay rates of the non-homogeneous transport equation}
Consider the equation
\begin{equation}
\label{DS-transportsimplied}
\frac{\partial g}{\partial t}+v.\nabla g=0,
\end{equation}
with the initial condition
\begin{equation}
\label{DS-transportsimpliedinitial}
g(0,x,v)=g_0(x,v).
\end{equation}
The energy of $(\ref{DS-transportsimplied})$ is then defined
\begin{equation}
\label{Energyg}
E_g(t)=\int_{\mathbb{T}^d}\int_V|g-g_\infty|^2dvdx,
\end{equation}
where 
\begin{equation}
\label{AsymProfilef}
g_\infty=\int_{\mathbb{T}^d}\int_Vg_0(x,v)dxdv.
\end{equation}
The following proposition is a direct consequence of Lemmas $\ref{Lemma0}$, $\ref{Lemma1}$ and $\ref{Lemma2}$.
\begin{proposition}\label{PropositionTransportExponentialDecays}
Suppose that there exist positive numbers $T_0$ and $\delta$ such that
\begin{equation}\label{TransportExpoDecay}
\forall t\geq T_0, \forall f_0 \in L^{\infty}(\mathbb{T}^d\times V): ~~~~ H_f(t)\leq \exp(-\delta t) H_f(0),
\end{equation}
then there exist a positive number $T_1$ and a nonnegative number $C$ such that
\begin{equation}\label{Transportobservability}
\int_0^{T_1}\int_{\mathbb{T}^d}\int_V\sigma (g-\bar{g})^2 dvdxdt\geq C\int_{\mathbb{T}^d}\int_V(g_0-g_\infty)^2dvdx,
\end{equation} 
for $g_0=f_0$.
\\Moreover, if there exist $T_1$ and $C$ such that $(\ref{Transportobservability})$ satisfies, then there exist $T_0$ and $\delta$ such that $(\ref{TransportExpoDecay})$ is true.
\end{proposition}
Theorem $\ref{TheoremTransportexponentialdecay2}$ is a direct consequence of Proposition $\ref{PropositionTransportExponentialDecays}$ and the following proposition.
\begin{proposition}\label{ProproTransportEnergy-Dampingwithweight} For $\sigma$ belongs to $H^1(\mathbb{T}^d)$, there exists a positive constant $T_0$ such that for $T>T_0$
\begin{equation}\label{TransportObservabilityWithWeight}
\int_0^T\int_{\mathbb{T}^d}\int_V\sigma|g-\bar{g}|^2dvdxdt \geq C(T) \int_0^T\int_{\mathbb{T}^d}|g_0-g_\infty|^2dxdt.
\end{equation}
\end{proposition}
\begin{proof}
Write $g$ under the form of Fourier series
$$g(x,v,t)=g_0(x-vt,v)=\sum_{n\in\mathbb{Z}^d}a_n(v)\exp(i2\pi n(x-vt)).$$
\\ The left hand side of $(\ref{TransportObservabilityWithWeight})$ becomes
\begin{eqnarray}\label{TransportProof1}\nonumber
& &\int_0^T\int_{\mathbb{T}^d}\int_V\sigma|g-\bar{g}|^2dvdxdt\\\nonumber
&=&\int_0^T\int_{\mathbb{T}^d}\int_V\sigma[\sum_{n\in\mathbb{Z}^d}(a_n(v)\exp(-i{2\pi}nvt)\exp(i{2\pi}nx)\\
& &-\int_Va_n(v)\exp(-i{2\pi}nvt)\exp(i{2\pi}nx)dv)]^2dvdxdt\\\nonumber
&=&\lim_{M\to\infty}\int_0^T\int_{\mathbb{T}^d}\int_V\sigma\left[[\sum_{|n|<M,n\ne0}|a_n(v)\exp(-i{2\pi}nvt)-\int_Va_n(v)\exp(-i{2\pi}nvt)dv|^2\right.\\\nonumber
& &+\sum_{|p|,|q|<M;p,q\ne0;p\ne q}\exp(i(p-q)\pi x)\times\\\nonumber
& &\times(a_p(v)\exp(-i{2\pi}pvt)-\int_Va_p(v)\exp(-i{2\pi}pvt)dv)\times\\\nonumber
& &\left.\times\overline{a_q(v)\exp(-i{2\pi}qvt)-\int_Va_q(v)\exp(-i{2\pi}qvt)dv} ~~\right]dvdxdt.
\end{eqnarray}
In the last sum, consider the integral
\begin{eqnarray}\label{TransportProof2}\nonumber
& &\int_0^T\int_V(a_p(v)\exp(-i{2\pi}pvt)-\int_Va_p(v)\exp(-i{2\pi}pvt)dv)\times\\
& &\times(\overline{a_q(v)\exp(-i{2\pi}qvt)-\int_Va_q(v)\exp(-i{2\pi}qvt)dv}) dvdt\\\nonumber
&=&\int_0^T\int_V\exp(i2(q-p)\pi vt)a_p(v)\overline{a_q(v)}dvdt\\\nonumber
& &-\int_0^T\int_V\exp(-i2p\pi vt){a_p(v)}dv\int_V\exp(i2q\pi vt)\overline{a_q(v)}dvdt.
\end{eqnarray}
Consider the first integral in the previous equality
\begin{eqnarray}\label{TransportProof3}\nonumber
& &\left|\int_0^T\int_V\exp(i2(q-p)\pi vt)a_p(v)\overline{a_q(v)}dvdt\right|\\\nonumber
&=&\left|\int_V\frac{\sin(\pi(p-q)vT)}{\pi(p-q)v}[-\cos(\pi(p-q)vT)-i\sin(\pi(p-q)vT)]a_p(v)\overline{a_q(v)}dv\right|\\
&\leq& \|a_p\|_{L^3} \|a_q\|_{L^3}\left(\int_V\left|\frac{\sin(\pi(p-q)vT)}{\pi(p-q)v}\right|^3dv\right)^{1/3}\\\nonumber
&\leq& C\|a_p\|_{L^2} \|a_q\|_{L^2}\left(\int_V\left|\frac{\sin(\pi(p-q)vT)}{\pi(p-q)v}\right|^3dv\right)^{1/3},
\end{eqnarray}
the last inequality follows from the fact that $a_p,a_q\in L^{\infty}(V)$. Moreover, $C$ is a global constant, which depends on the structure of the equation and the $L^\infty$ bound of $f_0$ as well as  of $f$, since $f$ satisfies the maximum principle.\\
Denote $$\mathfrak{m}=\sup_{|\sigma|=1,v\in\mathbb{R}^d}|\{|\pi\sigma v|\leq 1\}|,$$
and $|O|$ to be the Lebesgue measure of a set $O$ in $\mathbb{R}^d$, consider the integral
\begin{eqnarray*}\nonumber
& &\int_V\left|\frac{\sin(\pi(p-q)vT)}{\pi(p-q)v}\right|^3dv\\\nonumber
&=&\int_{\{|\pi(p-q)v|\leq\epsilon, v\in V\}}\left|\frac{\sin(\pi(p-q)vT)}{\pi(p-q)v}\right|^3dv+\\
& &+\int_{\{|\pi(p-q)v|>\epsilon, v\in V\}}\left|\frac{\sin(\pi(p-q)vT)}{\pi(p-q)v}\right|^3dv\\\nonumber
&\leq&T^3|{\{|\pi(p-q)v|\leq\epsilon, v\in V\}}|+\frac{1}{\epsilon^3}|{\{|\pi(p-q)v|>\epsilon, v\in V\}}|\\\nonumber
&\leq&T^3\frac{\mathfrak{m}\epsilon^d}{|p-q|^d}+\frac{1}{\epsilon^3}|V|\\\nonumber
&\leq&T^3\frac{\mathfrak{m}\epsilon}{|p-q|}+\frac{1}{\epsilon^3}.
\end{eqnarray*}
Due to Cauchy's inequality
\begin{eqnarray*}\nonumber
T^3\frac{\mathfrak{m}\epsilon}{|p-q|}+\frac{1}{\epsilon^3}&\geq& 4\left(\frac{\mathfrak{m}}{3}\right)^{3/4}\frac{T^{9/4}}{|p-q|^{3/4}},
\end{eqnarray*}
the previous inequality implies
\begin{eqnarray*}\nonumber
\left(\int_V\left|\frac{\sin(\pi(p-q)vT)}{\pi(p-q)v}\right|^3dv\right)^{1/3}\leq C\frac{T^{3/4}}{|p-q|^{1/4}},
\end{eqnarray*}
for a suitably chosen $\epsilon$ and $C$ is some positive constant. Combine this inequality with $(\ref{TransportProof3})$
\begin{equation}\label{TransportProof4}\
\left|\int_0^T\int_V\exp(i(q-p)\pi vt)a_p(v)\overline{a_q(v)}dvdt\right|\leq C\|a_p\|_{L^2} \|a_q\|_{L^2}\frac{T^{3/4}}{|p-q|^{1/4}},
\end{equation}
where $C$ is some positive constant.
\\\ Consider the following term in $(\ref{TransportProof2})$
\begin{eqnarray}\label{TransportProof5}\nonumber
& &\left|\int_0^T\left(\int_V\exp(-ip\pi vt){a_p(v)}dv\int_V\exp(iq\pi vt)\overline{a_q(v)}dv\right)dt\right|\\
& \leq &\int_0^T|\hat{a_p}(p t)||\hat{a_q}(q t)|dt\leq\frac{1}{\sqrt{|p||q|}}\|\hat{a_p}\|_{L^2}\|\hat{a_q}\|_{L^2}\\\nonumber
&\leq&\frac{1}{\sqrt{|p||q|}}\|{a_p}\|_{L^2}\|{a_q}\|_{L^2},
\end{eqnarray}
where $\hat{a_p}$, $\hat{a_q}$ are the Fourier transforms in $V$ of $a_p$ and $a_q$ with the assumption that the values of $a_p$, $a_q$ outside of $V$ are $0$.
\\ Combine $(\ref{TransportProof2})$, $(\ref{TransportProof4})$ and $(\ref{TransportProof5})$ to get
\begin{eqnarray}\label{TransportProof6}\nonumber
& &|\int_0^T\int_V[(a_p(v)\exp(-i{2\pi}pvt)-\int_Va_p(v)\exp(-i{2\pi}pvt)dv)\times\\
& &\times\overline{a_q(v)\exp(-i{2\pi}qvt)-\int_Va_q(v)\exp(-i{2\pi}qvt)dv}] dvdt|\\\nonumber
&\leq&\|a_p \|_{L^2}\|{a_q}\|_{L^2}\left(\frac{1}{\sqrt{|p||q|}}+C\frac{T^{3/4}}{|p-q|^{1/4}}\right).
\end{eqnarray}
\\Suppose that $|p_k-q_k|=\max\{|p_1-q_1|,\dots,|p_d-q_d|\}$,
\begin{eqnarray*}
& &\left|\int_{\mathbb{T}^d}\sigma(x)\exp(i2\pi(p-q)x)dx\right|\\\nonumber
& =&\left|-\int_{\mathbb{T}^d}\partial_{k}\sigma(x)\frac{\exp(i2\pi(p-q)x)}{i2\pi(p_k-q_k)}dx\right|\\\nonumber
& =&\int_{\mathbb{T}^d}|\partial_{k}\sigma(x)|\frac{1}{2\pi|p_k-q_k|}dx\\\nonumber
&\leq&C\|\sigma\|_{H^1}\frac{1}{|p-q|},
\end{eqnarray*}
where $C$ is some positive constant. Combine this inequality with $(\ref{TransportProof6})$ 
\begin{eqnarray}\label{TransportProof7}\nonumber
& &|\int_{\mathbb{T}^d}\int_0^T\int_V\sigma(a_p(v)\exp(-i{2\pi}pvt)-\int_Va_p(v)\exp(-i{2\pi}pvt)dv)\times\\
& &\times\overline{a_q(v)\exp(-i{2\pi}qvt)-\int_Va_q(v)\exp(-i{2\pi}qvt)dv} dvdtdx|\\\nonumber
&\leq&C\|\sigma\|_{H^1}\|a_p \|_{L^2}\|{a_q}\|_{L^2}\left(\frac{1}{\sqrt{|p||q|}|p-q|}+C\frac{T^{3/4}}{|p-q|^{5/4}}\right).
\end{eqnarray}
Since 
\begin{eqnarray*}\nonumber
& &\int_0^T\int_V|a_n(v)\exp(-i{2\pi}nvt)-\int_Va_n(v)\exp(-i{2\pi}nvt)dv|^2dvdt\\\nonumber
&=&\int_0^T\left(\int_V|a_n(v)|^2dv-\left|\int_Va_n(v)\exp(-i{2\pi}nvt)dv\right|^2\right)dt\\\nonumber
&=&\int_0^T\left(\int_V|a_n(v)|^2dv-|\hat{a_n}(nt)|^2\right)dt\\\nonumber
&=&T\int_V|a_n(v)|^2dv-\int_0^T|\hat{a_n}(nt)|^2dt\\\nonumber
&>&(T-1)\int_V|a_n(v)|^2dv,
\end{eqnarray*}
Inequality $(\ref{TransportProof7})$ implies
\begin{eqnarray*}\label{TransportProof1}\nonumber
& &\int_0^T\int_{\mathbb{T}^d}\int_V\sigma|g-\bar{g}|^2dvdxdt\\\nonumber
&\geq& \sum_{n\in\mathbb{R}^d,n\ne0}\frac{T}{2}\int_V|a_n(v)|^2dv,
\end{eqnarray*}
for $T$ large enough, which leads to $(\ref{TransportObservabilityWithWeight})$.
\end{proof}
\section{Decay rates of the special transport equation}
\subsection{The Observability Inequality}
\begin{proposition}\label{3ProproTransportEnergy-Dampingwithweight} For $\sigma$ belongs to $C^\infty(\mathbb{T}^d)$, there exists a positive constant $T_0$ such that for $T>T_0$
\begin{equation}\label{3TransportObservabilityWithWeight}
\int_0^T\int_{\mathbb{T}^d}\int_V\sigma(1-\Delta_{x})^{-\epsilon}\sigma|g-\bar{g}|^2dvdxdt \geq C(T,\sigma) \sum_{n\in \mathbb{Z}^d}\frac{\int_{\mathbb{R}^d}|A_n(v)|^2dv}{(1+|n|^2)^\epsilon}.
\end{equation}
\end{proposition}
\begin{proof} For $n$ in $\mathbb{Z}^d$, define
$$A_n(v)=\int_{(0,1)^d}\tilde{g}_0(x,v)\exp(-in2\pi x)dx.$$
For the sake of simplicity we only give the proof for the case where $g_0(x,v)=g_0(x)$, which means $A_n$ does not depend on $v$ for all $n$ in $\mathbb{Z}^d$. The case where $A_n$ depends on $v$ could be treated with the technical tricks of the previous section.
\\  Write $\tilde{g}$ under the form of Fourier series:
\begin{eqnarray*}
{g}(x,v,t)={g}_0(x-vt,v)=\sum_{n\in{\mathbb{Z}^d}}A_n\exp(i2\pi n(x-vt)),
\end{eqnarray*}
which deduces
\begin{eqnarray}\label{SpecialProof1}\nonumber
& &\int_{\mathbb{T}^d}\int_V\int_0^T\sigma(1-\Delta_{x})^{-\epsilon}\sigma|g-\bar{g}|^2dtdx\\\nonumber
&=&\lim_{m\to\infty}\int_{\mathbb{T}^d}\left\{\sum_{n\in\mathbb{Z}^d;|n|\leq m;n\ne 0}|A_n|^2|(1-\Delta_x)^{-\frac{\epsilon}{2}}\sigma\exp(i2\pi nx)|^2dx\times\right.\\\nonumber
& & \times\int_0^T\int_{V}|\exp(-i2\pi nvt)-\int_V\exp(-i2\pi nvt)dv|^2dvdt\\
& &+\sum_{p,q\in\mathbb{Z}^d;|p|,|q|\leq m;p\ne q; p,q\ne0}(1-\Delta_x)^{-\frac{\epsilon}{2}}\sigma\exp(i2\pi px)A_p\times\\\nonumber
& &\overline{(1-\Delta_x)^{-\frac{\epsilon}{2}}\sigma\exp(i2\pi qx)A_q}\int_0^T\int_V[(\exp(-i2\pi pvt)-\int_V\exp(-i2\pi pvt)dv)\times\\\nonumber
& &\left.\times\overline{\exp(-i2\pi qvt)-\int_V\exp(-i2\pi qvt)dv}]dvdt\right\}dx.
\end{eqnarray}
Similar as in the previous section, we have
\begin{eqnarray*}\label{SpecialProof2}
\int_0^T\int_{V}|\exp(-i2\pi nvt)-\int_V\exp(-i2\pi nvt)dv|^2dvdt\geq\frac{T}{2},
\end{eqnarray*}
for $T$ large enough, and
\begin{eqnarray*}\label{SpecialProof3}
& &\left|\int_0^T\int_V\left((\exp(-i2\pi pvt)-\int_V\exp(-i2\pi pvt)dv)\times\right.\right.\\\nonumber
& &\left.\left.\overline{\exp(-i2\pi qvt)-\int_V\exp(-i2\pi qvt)dv}\right)dvdt\right|\\\nonumber
&=&\left|\int_0^T\int_V\exp(i2\pi (q-p)vt)dvdt\right.\\\nonumber
& &\left.-\int_0^T(\int_V\exp(-i2\pi pvt)dv\int_V\exp(i2\pi qvt)dv)dt\right|\\\nonumber
&\leq& C\left(\frac{T^{1/2}}{|p-q|^{1/2}}+\frac{1}{\sqrt{|p||q|}}\right),
\end{eqnarray*}
where $C$ is some positive constant.
Consider the sum
\begin{eqnarray*}
&&\sum_{n\in\mathbb{Z}^d;|n|\leq m;|n|\leq M;n\ne0}|A_n|^2\int_{\mathbb{T}^d}|(1-\Delta_x)^{-\frac{\epsilon}{2}}\sigma\exp(i2\pi nx)|^2dx\times\\\nonumber
& & \times\int_0^T\int_{V}|\exp(-i2\pi nvt)-\int_V\exp(-i2\pi nvt)dv|^2dvdt\\
& \geq & \sum_{n\in\mathbb{Z}^d;|n|\leq m;|n|\leq M;n\ne0}\frac{T}{2}|A_n|^2\int_{\mathbb{T}^d}|(1-\Delta_x)^{-\frac{\epsilon}{2}}\sigma \exp(i2\pi nx)|^2dx\\
& \geq & TC(M)\sum_{n\in\mathbb{Z}^d;|n|\leq m;|n|\leq M;n\ne0}\frac{|A_n|^2}{(1+n^2)^\frac{\epsilon}{2}},
\end{eqnarray*}
and
\begin{eqnarray*}
&&\sum_{n\in\mathbb{Z}^d;m\geq|n|> M}|A_n|^2\int_{\mathbb{T}^d}|(1-\Delta_x)^{-\frac{\epsilon}{2}}\sigma\exp(i2\pi nx)|^2dx\times\\\nonumber
& & \times\int_0^T\int_{V}|\exp(-i2\pi nvt)-\int_V\exp(-i2\pi nvt)dv|^2dvdt\\
& \geq & CT\sum_{n\in\mathbb{Z}^d;m\geq|n|> M}|A_n|^2\int_{\mathbb{T}^d}|\sigma(1-\Delta_x)^{-\frac{\epsilon}{2}} \exp(i2\pi n(x-vt))|^2dx\\
& &-CT\sum_{n\in\mathbb{Z}^d;m\geq|n|> M}|A_n|^2\int_{\mathbb{T}^d}|[\sigma,(1-\Delta_x)^{-\frac{\epsilon}{2}}] \exp(i2\pi n(x-vt))|^2dx\\
& \geq & T C\sum_{n\in\mathbb{Z}^d;m\geq|n|> M}\frac{|A_n|^2}{(1+n^2)^\frac{\epsilon}{2}}\int_{\mathbb{T}^d}\sigma^2 dx\\
& &-TC'(\sigma)\sum_{n\in\mathbb{Z}^d;m\geq|n|> M}\frac{|A_n|^2}{(1+n^2)^{\frac{\epsilon+1}{2}}},
\end{eqnarray*}
where $M$, $C(M)$, $C$, $C'$ are  positive constants and the last inequality follows from the fact that the pseudodifferential operator $(1-\Delta_x)^{-\frac{\epsilon}{2}}$ is of order $\epsilon$ and the commutator $[\sigma,(1-\Delta_x)^{-\frac{\epsilon}{2}}]$ is of order $\epsilon+1$. Combine these two sums to get
\begin{eqnarray}\label{SpecialProof4}\nonumber
&&\sum_{n\in\mathbb{Z}^d}|A_n|^2\int_{\mathbb{T}^d}|(1-\Delta_x)^{-\frac{\epsilon}{2}}\sigma\exp(i2\pi nx)|^2dx\times\\
& & \times\int_0^T\int_{V}|\exp(-i2\pi nvt)-\int_V\exp(-i2\pi nvt)dv|^2dvdt\\\nonumber
& \geq & TC(\sigma)\sum_{n\in\mathbb{Z}^d}\frac{|A_n|^2}{(1+n^2)^\frac{\epsilon}{2}},
\end{eqnarray}
for $T$ and $m$ large enough, where $C(\sigma)$ is a constant depending on $\sigma$.
\\ Now, consider the term
\begin{eqnarray}\label{SpecialProof5}\nonumber
& &\left|\sum_{p,q\in\mathbb{Z}^d;|p|,|q|\leq m;p\ne q}(1-\Delta_x)^{-\frac{\epsilon}{2}}\sigma\exp(i2\pi px)A_p\times\right.\\\nonumber
& &\overline{(1-\Delta_x)^{-\frac{\epsilon}{2}}\sigma\exp(i2\pi qx)A_q}\int_0^T\left[\int_V(\exp(-i2\pi pvt)-\int_V\exp(-i2\pi pvt)dv)\times\right.\\
& &\left.\left.\times\overline{\exp(-i2\pi qvt)-\int_V\exp(-i2\pi qvt)dv}\right]dvdt\right|\\\nonumber
&\leq&\sum_{p,q\in\mathbb{Z}^d;|p|,|q|\leq m;p\ne q}C\frac{|A_p|}{(1+|p|^2)^{\frac{\epsilon}{2}}}
\frac{|A_q|}{(1+|q|^2)^{\frac{\epsilon}{2}}}T^{1/2}\left(\frac{T^{\frac{1}{2}}}{|p-q|^{3/2}}+\frac{1}{|p-q|\sqrt{|p||q|}}\right)\\\nonumber
&\leq&\sum_{p\in\mathbb{Z}^d;|p|\leq m}\left(\frac{|A_p|^2}{(1+|p|^2)^{{\frac{\epsilon}{2}}}}\sum_{q\in\mathbb{Z}^d;|q|\leq m;q\ne p}\left(\frac{1+|p|^2}{1+|q|^2}\right)^{\frac{\epsilon}{2}}\left[\frac{T^{1/2}}{|p-q|^{3/2}}+\frac{1}{|p-q|\sqrt{|p||q|}}\right)\right],
\end{eqnarray}
where the last sum converges for $\epsilon$ small enough.
\\ Combine $(\ref{SpecialProof1})$, $(\ref{SpecialProof4})$ and $(\ref{SpecialProof5})$ to get
\begin{eqnarray*}
\int_0^T\int_{\mathbb{T}^d}\int_V\sigma|g-\bar{g}|^2dvdxdt \geq C(T,\sigma) \sum_{n\in \mathbb{Z}^d}\frac{\int_{\mathbb{R}^d}|A_n(v)|^2dv}{(1+|n|^2)^\frac{\epsilon}{2}}.
\end{eqnarray*}
\end{proof}
\subsection{Convergence to Equilibrium: Proof of Theorem $\ref{TheoremTransportexponentialdecay3}$}
{\bf Step 1:} The boundedness of $||(1-\Delta_x)^{\frac{\epsilon}{2}} f||_{L^2}$.\\
Use $(1-\Delta_x)^\epsilon$ as a test function in $(\ref{DS-transportspecial})$ to get
$$\int_{\mathbb{T}^d}\partial_t f(1-\Delta_x)^\epsilon f+\int_{\mathbb{T}^d}v\partial_x f(1-\Delta_x)^\epsilon f=\int_{\mathbb{T}^d}\sigma(1-\Delta_x)^{-\epsilon}\sigma(\bar{f}-f)(1-\Delta_x)^\epsilon f.$$
This leads to 
$$\partial_t\|(1-\Delta_x)^{\frac{\epsilon}{2}} f\|^2_{L^2}\leq C\|f\|^2_{L^2},$$
which means
$$\|(1-\Delta_x)^{\frac{\epsilon}{2}} f(t)\|^2_{L^2}\leq Ct\|f(0)\|^2_{L^2}+\|(1-\Delta_x)^{\frac{\epsilon}{2}} f(0)\|_{L^2}^2.$$
{\bf Step 2:} The polynomial convergence.\\
The previous proposition and Lemma $\ref{Lemma1}$ imply
\begin{equation}\label{4ObservabilityHuWithWeight}
H_f(0)-H_f(T) \geq C(T,\sigma) \sum_{n\in \mathbb{Z}^d}\frac{|A_n|^2}{(1+|n|^2)^\frac{\epsilon}{2}}.
\end{equation}
Let $k_1$, $k_2$ and $k_3$ be positive numbers satisfying $-2\epsilon k_1+ k_2k_3>0$ and $k_2<\epsilon$. According to the Jensen Inequality
\begin{eqnarray*}
& &\left(\frac{\sum_{n\in\mathbb{Z}^d} \frac{|A_n|^2}{(1+|n|^2)^\frac{\epsilon}{2}}}{\sum_{n\in\mathbb{Z}^d} |A_n|^2}\right)^{k_1}\left(\frac{\sum_{n\in\mathbb{Z}^d} |A_n|^2 |n|^{k_2}}{\sum_{n\in\mathbb{Z}^d} |A_n|^2}\right)^{k_3}\\
&\geq& \left(\sum_{n\in\mathbb{Z}^d}\frac{|A_n|^2((1+|n|^2)^{-\frac{\epsilon k_1}{k_1+k_3}}|n|^{\frac{k_2k_3}{k_1+k_3}})}{\sum_{n\in\mathbb{Z}^d} |A_n|^2}\right)^{k_1+k_3}\\
&\geq& C\left(\sum_{n\in\mathbb{Z}^d}\frac{|A_n|^2}{\sum_{n\in\mathbb{Z}^d} |A_n|^2}\right)^{k_1+k_3}\\
&\geq&C,
\end{eqnarray*}
where $C$ is some positive constant, which yields
\begin{equation}\label{4H-estimate}
\sum_{n\in\mathbb{Z}^d} \frac{|A_n|^2}{(1+|n|^2)^\frac{\epsilon}{2}}\geq C\left(\sum_{n\in\mathbb{Z}^d} |A_n|^2\right)\left(\frac{\sum_{n\in\mathbb{Z}^d} |A_n|^2}{\sum_{n\in\mathbb{Z}^d} |A_n|^2 |n|^{k_2}}\right)^{\frac{k_3}{k_1}},
\end{equation}
for some positive constant $C$.\\
Denote $\tilde{f}=f-f_0$ and 
$$M((l-1)T)=\sum_{n\in\mathbb{Z}^d} |\hat{f(lT)}(n)|^2|n|^{k_2},$$
for $l\in\mathbb{N}\{0\}$, where $\hat{f(lT)}$ is the Fourier transform in $x$ of $f(lT)$.\\
Inequalities $(\ref{4ObservabilityHuWithWeight})$ and $(\ref{4H-estimate})$ imply
\begin{equation}\label{4RelationHu0T}
H_f(0)-CH_f(0)\left(\frac{ H_f(0)}{M(0)}\right)^{\frac{k_3}{k_1}}\geq H_f(T).
\end{equation}
Since the energy $H_f$ is decreasing, $(\ref{4RelationHu0T})$ deduces
\begin{equation}\label{4RelationHulT}
H_f(lT)-CH_f(lT)\left(\frac{H_f(lT)}{M(lT)}\right)^{\frac{k_3}{k_1}}\geq H_f((l+1)T),
\end{equation}
for all $l$ in $\mathbb{N}\cup\{0\}$. Step $1$ implies $M(lT)\leq lTC$, where C is a positive constant, which together with Inequality $(\ref{4RelationHulT})$ implies
 \begin{equation}\label{4RelationHulTb}
H_f(lT)-CH_f(lT)\left(\frac{H_f(lT)}{ lTC}\right)^{\frac{k_3}{k_1}}\geq H_f((l+1)T).
\end{equation}
Put $$\mathcal{E}_l=\frac{H_f(lT)}{ lTC},$$
Inequality $(\ref{4RelationHulTb})$ yields
$$\mathcal{E}_{l+1}\leq\mathcal{E}_l-C\mathcal{E}_{l+1}^{\frac{k_3}{k_1}+1},$$
where $C$ is some positive constant. 
According to Lemma $\ref{LemmaAmmari}$
$$H_f(lT)\leq C l\left(\frac{1}{l+1}\right)^{\frac{k_1}{k_3}}H_f(0),$$
where $C$ is some positive constant. Let $\frac{k_1}{k_2}$ tend to $\infty$ we get the theorem.
\section{Decay rates of the linearized Boltzmann equation}
Let $g$ be the solution of 
\begin{equation}\label{BoltzmannHomo}
\partial_t g+v\partial_x g=0,
\end{equation}
with the initial datum
$$g(0,x,v)=f_0(x,v),$$
where $f_0(x,v)$ is the initial datum of $(\ref{BoltzmannLinearized})$. For the sake of simplicity, we suppose that
\begin{equation}
\int_{\mathbb{R}^d}f_0dv=0.
\end{equation}
\subsection{The Observability Inequality}
Similar as in the previous sections, we prove
\begin{proposition}\label{BPropObservability} There exists a constant $T_*$, depending on the structure of the equation, such that for all $T>T_*$
\begin{eqnarray}\label{BoltzmannObservability}
& &\int_0^{T}\int_{\mathbb{T}^d\times\mathbb{R}^d\times\mathbb{R}^d\times\mathbb{S}^{d-1}}\mathcal{B}(|v_*-v|,\omega)\mu_*\mu\times\\\nonumber
& &\times[g'_*{\mu'_*}^{-1/2}+g'\mu'^{-1/2}-g_*\mu^{-1/2}_*-g\mu^{-1/2}]^2d\sigma dv_*dvdx\\\nonumber
& \geq &C\int_{\mathbb{T}^d\times\mathbb{R}^d}(|v|+1)^{\alpha}|f_0|^2dxdv.
\end{eqnarray}
\end{proposition}
\begin{proof}
Since $g$ is a solution of $(\ref{BoltzmannHomo})$, it could be written under the form
\begin{eqnarray*}
g(t,x,v)=g_0(x-vt,v)=\sum_{n\in\mathbb{Z}^d}A_n(v)\exp(i2\pi n(x-vt)),
\end{eqnarray*}
this implies
\begin{eqnarray}\label{BCollisionKernelFourier}
& &\int_0^{T}\int_{\mathbb{T}^d\times\mathbb{R}^d\times\mathbb{R}^d\times\mathbb{S}^{d-1}}\mathcal{B}(|v_*-v|,\omega)\mu_*\mu\times\\\nonumber
& &\times[g'_*{\mu'_*}^{-1/2}+g'\mu'^{-1/2}-g\mu^{-1/2}-g_*\mu^{-1/2}_*]^2d\omega dv_*dvdxdt\\\nonumber
&=&\sum_{n\in\mathbb{Z}^d}\int_0^{T}\int_{\mathbb{R}^d\times\mathbb{R}^d\times\mathbb{S}^{d-1}}\mathcal{B}\mu_*\mu\left|{A_n}'_*{\mu'_*}^{-1/2}\exp(-i2\pi nv'_*t)+{A_n}'\mu'^{-1/2}\exp(-i2\pi nv't)\right.\\\nonumber
& &\left.-{A_n}_*\mu^{-1/2}_*\exp(-i2\pi nv_*t)-{A_n}\mu^{-1/2}\exp(-i2\pi nvt)\right|^2d\omega dv_*dvdt\\\nonumber
&=&4\sum_{n\in\mathbb{Z}^d}\int_0^{T}\int_{\mathbb{R}^d\times\mathbb{R}^d\times\mathbb{S}^{d-1}}\mathcal{B}\mu_*\mu[-{A_n}'_*{\mu'_*}^{-1/2}\overline{A_n}\mu^{-1/2}\exp(i2\pi n(v-v'_*)t)-\\\nonumber
& &-{A_n}'\mu'^{-1/2}\overline{A_n}\mu^{-1/2}\exp(i2\pi n(v-v')t)\\\nonumber
& &+{A_n}_*\mu^{-1/2}_*\overline{A_n}\mu^{-1/2}\exp(i2\pi n(v-v_*)t)+|{A_n}\mu^{-1/2}|^2]d\omega dv_*dvdt.
\end{eqnarray}
Using the same technique as in \cite{Grad:1958:PKT}, \cite{Grad:1963:ATB} and \cite{MouhotStrain:2007:SGC}, we  consider the components of the last integral of $(\ref{BCollisionKernelFourier})$ separately. 
{\\\bf Part 1:} Consider the dominant component of $(\ref{BCollisionKernelFourier})$
\begin{eqnarray}\label{BComponent1}
& &\int_0^{T}\int_{\mathbb{R}^d\times\mathbb{R}^d\times\mathbb{S}^{d-1}}\mathcal{B}\mu_*\mu|{A_n}\mu^{-1/2}|^2d\omega dv_*dvdt\\\nonumber
&=&T\int_{\mathbb{R}^d\times\mathbb{R}^d\times\mathbb{S}^{d-1}}\mathcal{B}\mu_*|{A_n}|^2d\omega dv_*dv\\\nonumber
&\geq&TC\int_{\mathbb{R}^d}(|v|+1)^\alpha|{A_n}|^2dv,
\end{eqnarray}
where $C$ is some positive constant.
{\\\bf Part 2:} Consider the second component of $(\ref{BCollisionKernelFourier})$
\begin{eqnarray}\label{BComponent2pre}\nonumber
& &\left|\int_0^{T}\int_{\mathbb{R}^d\times\mathbb{R}^d\times\mathbb{S}^{d-1}}\mathcal{B}\mu^{1/2}\mu_*^{1/2}{A_n}_*\overline{A_n}\exp(i2\pi n(v-v_*)t)d\omega dv_*dvdt\right|\\
&\leq&\int_{\mathbb{R}^d\times\mathbb{R}^d\times\mathbb{S}^{d-1}}\mathcal{B}\mu^{1/2}\mu_*^{1/2}|{A_n}_*||{A_n}|\frac{|\sin(\pi n(v-v_*)T)|}{|\pi n(v-v_*)|}d\omega dv_*dv\\\nonumber
&\leq& \frac{1}{2}\int_{\mathbb{R}^d\times\mathbb{R}^d\times\mathbb{S}^{d-1}}\mathcal{B}\mu^{1/2}\mu_*^{1/2}{|{A_n}|^2}\frac{|\sin(\pi n(v-v_*)T)|}{|\pi n(v-v_*)|}d\omega dv_*dv.
\end{eqnarray}
The kernel of $(\ref{BComponent2pre})$ could be bounded in the following way
\begin{eqnarray}\label{BComponent2Kernel}
& &\int_{\mathbb{R}^d\times\mathbb{S}^{d-1}}\mathcal{B}\mu_*^{1/2}\mu^{1/2}\frac{|\sin(\pi n(v-v_*)T)|}{|\pi n(v-v_*)|}d\omega dv_*\\\nonumber
&\leq&|\mathbb{S}^{d-1}|^{1/2}\left(\int_{\mathbb{R}^d\times\mathbb{S}^{d-1}}\mathcal{B}^2(\mu\mu_*)^{1/2}d\omega dv_*\right)^{1/2}\times\\\nonumber
& &\times\left(\int_{\mathbb{R}^d}\frac{|\sin(\pi n(v-v_*)T)|^2}{|\pi n(v-v_*)|^2}(\mu\mu_*)^{1/2}dv_*\right)^{1/2}\\\nonumber
&\leq& C(|v|+1)^{\beta-2/3}\left(\int_{\mathbb{R}^d}\frac{|\sin(\pi n(v-v_*)T)|^2}{|\pi n(v-v_*)|^2}(\mu\mu_*)^{1/2}dv_*\right)^{1/2},
\end{eqnarray}
where $C$ is some positive constant. 
\\ In order to estimate the last integral of $(\ref{BComponent2Kernel})$, let $\epsilon$ be a positive constant, we consider two cases.
\\ For $|n(v-v_*)|<\epsilon$,
\begin{eqnarray}\label{BComponent2Kernel1}
& &\left(\int_{\{|n(v-v_*)|\}<\epsilon}\frac{|\sin(\pi n(v-v_*)T)|^2}{|\pi n(v-v_*)|^2}(\mu\mu_*)^{1/2}dv_*\right)^{1/2}\\\nonumber
&\leq&T\left(\int_{\{|n(v-v_*)|\}<\epsilon}(\mu\mu_*)^{1/2}dv_*\right)^{1/2}\\\nonumber
&\leq& TC(\epsilon),
\end{eqnarray}
where $C(\epsilon)$ tends to $0$ as $\epsilon$ tends to $0$.
\\ For $|n(v-v_*)|>\epsilon$,
\begin{eqnarray}\label{BComponent2Kernel2}
& &\left(\int_{\{|n(v-v_*)|\}>\epsilon}\frac{|\sin(\pi n(v-v_*)T)|^2}{|\pi n(v-v_*)|^2}(\mu\mu_*)^{1/2}dv_*\right)^{1/2}\\\nonumber
&\leq&\left(\int_{\{|n(v-v_*)|\}>\epsilon}(\mu\mu_*)^{1/2}\frac{1}{\epsilon^2}dv_*\right)^{1/2}\\\nonumber
&\leq& \frac{C}{\epsilon},
\end{eqnarray}
where $C$ is some positive constant. 
 Inequalities $(\ref{BComponent2Kernel})$, $(\ref{BComponent2Kernel1})$ and $(\ref{BComponent2Kernel2})$ then imply
\begin{eqnarray}\label{BComponent2}\nonumber
& &\left|\int_0^{T}\int_{\mathbb{R}^d\times\mathbb{R}^d\times\mathbb{S}^{d-1}}\mathcal{B}\mu^{1/2}\mu_*^{1/2}{A_n}_*\overline{A_n}\exp(i2\pi n(v-v_*)t)d\omega dv_*dvdt\right|\\
&\leq&  \min\{TC(\epsilon), \frac{C}{\epsilon}\}\int_{\mathbb{R}^d}(|v|+1)^{\beta-2/3}|A_n(v)|^2dv.
\end{eqnarray}
{\\\bf Part 3:} Consider the last components of $(\ref{BCollisionKernelFourier})$, by the change of variables $\omega\to$ $-\omega$
\begin{eqnarray}\label{BComponent3pre}
&&I:=\\\nonumber
&=&\left|\int_0^{T}\int_{\mathbb{R}^d\times\mathbb{R}^d\times\mathbb{S}^{d-1}}\mathcal{B}\mu_*\mu[-{A_n}'_*{\mu'_*}^{-1/2}\overline{A_n}\mu^{-1/2}\exp(i2\pi n(v-v'_*)t)-\right.\\\nonumber
& &\left.-{A_n}'\mu'^{-1/2}\overline{A_n}\mu^{-1/2}\exp(i2\pi n(v-v')t)]d\omega dv_*dvdt\right|\\\nonumber
&=&\left|\int_0^{T}\int_{\mathbb{R}^d\times\mathbb{R}^d\times\mathbb{S}^{d-1}}2\mathcal{B}\mu_*\mu{A_n}'{\mu'}^{-1/2}\overline{A_n}\mu^{-1/2}\exp(i2\pi n(v-v')t)d\omega dv_*dvdt\right|\\\nonumber
&\leq&\int_{\mathbb{R}^d\times\mathbb{R}^d\times\mathbb{S}^{d-1}}2|v-v_*|^\beta|v-v'|^{d-2}|{A_n}'||{A_n}|\frac{|\sin(\pi n(v-v')T)|}{|\pi n(v-v')|}\mu_*^{1/2}{\mu_*'}^{1/2}d\omega dv_*dv,
\end{eqnarray}
the last inequality is derived by taking the integral in time. 
\\ Denote
$$K^*:=\int_{\mathbb{R}^d\times\mathbb{S}^{d-1}}2|v-v_*|^\beta|v-v'|^{d-2}|{A_n}'|\frac{|\sin(\pi n(v-v')T)|}{|\pi n(v-v')|}\mu_*^{1/2}{\mu_*'}^{1/2}d\omega dv_*,$$
and for $\omega$ fixed perform the following changes of variables on $K^*$: $v_*\to V=v_*-v$ and $V=r\omega+z$ with $z\in \omega^{\bot}$. Since the Jacobians of the two changes of variables are $1$,
$$K^*=\int_{\mathbb{R}^d\times\mathbb{S}^{d-1}}2r^{d-2}|{A_n(v+r\omega)}|\frac{|\sin(\pi rT n.\omega )|}{|\pi r n.\omega |}\left(\int_{\omega^{\bot}}(\mu_*{\mu_*'})^{1/2}|r\omega+z|^{\beta}dz\right)d\omega dr.$$
Now, make the change of variable $(r,\omega)\to W=r\omega$. The Jacobian of this change of variables is $2r^{-d+1}$.
$$K^*=\int_{\mathbb{R}^d\times\mathbb{S}^{d-1}}4|{A_n(v+W)}||W|^{-1}\frac{|\sin(\pi T n.W )|}{|\pi  n.W |}\left(\int_{W^{\bot}}(\mu_*{\mu_*'})^{1/2}|W+z|^\beta dz\right.)dW.$$
Since
\begin{eqnarray*}
|v_*|^2+|v_*'|^2 &=&|v+W+z|^2+|v+z|^2\\
&=&\frac{1}{2}|W+2(v+z)|^2+\frac{1}{2}|W|^2\\
&=&\frac{1}{2}|W+2(v.\omega)\omega|^2+2|z+v-(v.\omega)\omega|^2+\frac{1}{2}|W|^2,
\end{eqnarray*}
then
\begin{eqnarray*}
(\mu_*\mu_*')^{1/2}&=&(2\pi)^{-d/2}\exp\left(-\frac{|W|^2}{8}-\frac{|z+v-(v.\omega)\omega|^2}{2}-\frac{|W+2(v.\omega)\omega|^2}{8}\right),
\end{eqnarray*}
which implies
\begin{eqnarray*}
K^*& = &\int_{\mathbb{R}^d\times\mathbb{S}^{d-1}}4(2\pi)^{-d/2}|{A_n(v+W)}||W|^{-1}\exp(-\frac{|W|^2}{8}-\frac{|W+2(v.\omega)\omega|^2}{8})\\
& &\times\frac{|\sin(\pi T n.W )|}{|\pi  n.W |}\left(\int_{W^{\bot}}\exp(-\frac{|z+v-(v.\omega)\omega|^2}{2})|W+z|^\beta dz\right)dW.
\end{eqnarray*}
Define
\begin{eqnarray*}
K&:= &4(2\pi)^{-d/2}|v'-v|^{-1}\exp(-\frac{|v'-v|^2}{8}-\frac{|v'-v+2(v.\omega)\omega|^2}{8})\frac{|\sin(\pi T n.(v'-v) )|}{|\pi  n.(v'-v) |}\\
& &\times\left(\int_{\omega^{\bot}}\exp(-\frac{|z+v-(v.\omega)\omega|^2}{2})|v'-v+z|^\beta dz\right),
\end{eqnarray*}
then 
\begin{equation}\label{BComponent3prepre}
I\leq \int_{\mathbb{R}^d\times\mathbb{R}^d\times\mathbb{S}^{d-1}}K|A_n(v')||A_n(v)|d\omega dv'dv.
\end{equation}
Now, consider the integral in $z$ in the kernel $K$
\begin{eqnarray}\label{BComponent3K}
& &\int_{\omega^{\bot}}\exp(-\frac{|z+v-(v.\omega)\omega|^2}{2})|v'-v+z|^\beta dz\\\nonumber
&=& \int_{\omega^{\bot}}\exp(-\frac{|\overline{z}|^2}{2})|v'-v+\overline{z}-(v-(v.\omega)\omega)|^\beta dz\\\nonumber
&\leq& C(1+|v'-v-(v-(v.\omega)\omega)|)^\beta,
\end{eqnarray}
since $\beta>-(d-1)$, the integral is well-defined. Let $s$ be a real number, according to the inequality
$$(1+|\zeta'|)^s\leq C(1+|\zeta|)^s(1+|\zeta'-\zeta|)^{|s|},$$
the following estimate follows from $(\ref{BComponent3K})$
\begin{eqnarray}\label{BComponent3Ks}
& &\int_{\mathbb{R}^d}K(1+|v'|)^s dv'\\\nonumber
&\leq & C(1+|v|)^s\int_{\mathbb{R}^d}|v'-v|^{-1}\exp(-\frac{|v'-v|^2}{8}-\frac{|v'-v+2(v.\omega)\omega|^2}{8})\\\nonumber
& & \times\frac{|\sin(\pi T n.(v'-v) )|}{|\pi  n.(v'-v) |}(1+|v'-v-(v-(v.\omega)\omega)|)^\beta(1+|v'-v|)^{|s|}dv'\\\nonumber
&\leq & C(1+|v|)^s\left(\int_{\mathbb{R}^d}\frac{|\sin(\pi T n.(v'-v) )|^{3}}{|\pi  n.(v'-v) |^{3}}\exp(-\frac{|v'-v|^2}{8})dv'\right)^{1/3}\times\\\nonumber
& &\times \left(\int_{\mathbb{R}^d}|v'-v|^{-3/2}\exp(-\frac{|v'-v|^2}{8}-\frac{3|v'-v+2(v.\omega)\omega|^2}{16})\right.\\\nonumber
& &\left. \times(1+|v'-v-(v-(v.\omega)\omega)|)^{{3/2}\beta}(1+|v'-v|)^{{3/2}|s|}dv'\right)^{2/3}\\\nonumber
&\leq & (1+|v|)^s\min\{TC(\epsilon),\frac{C}{\epsilon}\}\left(\int_{\mathbb{R}^d}\exp(-\frac{|v'-v|^2}{8}-\frac{3|v'-v+2(v.\omega)\omega|^2}{16})\right.\\\nonumber
& &\left. \times C|v'-v|^{-{3/2}}(1+|v'-v-(v-(v.\omega)\omega)|)^{{3/2}\beta}(1+|v'-v|)^{{3/2}|s|}dv'\right)^{2/3},
\end{eqnarray}
the last inequality is obtained by the same argument that we use in Part 2. Now, we consider two cases $\beta\geq 0$ and $\beta<0$. 
{\\\it Case 1:} $\beta\geq 0$.
\begin{eqnarray*}\label{BComponent3KsBetaP}
& &\int_{\mathbb{R}^d}K(1+|v'|)^s dv'\\\nonumber
&\leq & (1+|v|)^s\min\{TC(\epsilon),\frac{C}{\epsilon}\}\left(\int_{\mathbb{R}^d}\exp(-\frac{|v'-v|^2}{8}-\frac{3|v'-v+2(v.\omega)\omega|^2}{16})\right.\\\nonumber
& &\left. \times C|v'-v|^{-{3/2}}(1+|v-(v.\omega)\omega|)^{{3/2}\beta}(1+|v'-v|)^{{3/2}|s|+{3/2}\beta}dv'\right)^{2/3}\\\nonumber
&\leq & (1+|v|)^{s+\beta}\min\{TC(\epsilon),\frac{C}{\epsilon}\}\left(\int_{\mathbb{R}^d}\exp(-\frac{|v'-v|^2}{8}-\frac{3|v'-v+2(v.\omega)\omega|^2}{16})\right.\\\nonumber
& &\left. \times C|v'-v|^{-{3/2}}(1+|v'-v|)^{{3/2}|s|+{3/2}\beta}dv'\right)^{2/3}.
\end{eqnarray*}
Denote 
$$J_1:=\int_{\mathbb{R}^d}\exp(-\frac{|v'-v|^2}{8}-\frac{3|v'-v+2(v.\omega)\omega|^2}{16})|v'-v|^{-{3/2}}(1+|v'-v|)^{{3/2}|s|+{3/2}\beta}dv'.$$
Perform the changes of variables $V\to u=v'-v$ and $u=r\omega$, $r\in\mathbb{R}_+$, $\omega\in S^{d-1}$, and choose $v$ as the north pole vector in the angle parametrization
\begin{eqnarray*}
J_1 & = &|S^{d-2}|\int_0^{\infty}r^{d-{5/2}}(1+r)^{{3/2}|s|+{3/2}\beta}\exp(-\frac{r^2}{8})\times\\
& &\times\int_0^\pi\exp(-\frac{3(r+2|v|\cos\varphi)^2}{16})\sin^{d-2}(\varphi)d\varphi dr.
\end{eqnarray*}
For the case $d\geq 3$, since $\sin^{d-2}(\varphi)\leq\sin(\varphi)$,
\begin{eqnarray*}
J_1 &\leq&|S^{d-2}|\int_0^{\infty}r^{d-{5/2}}(1+r)^{{3/2}|s|+{3/2}\beta}\exp(-\frac{r^2}{8})\times\\
& &\times\int_0^\pi\exp(-\frac{3(r+2|v|\cos\varphi)^2}{16})\sin(\varphi)d\varphi dr.
\end{eqnarray*}
Now, make the change of variables $y=r+2|v|\cos(\varphi)$ in the $\varphi$ integral to get
\begin{eqnarray*}
J_1&\leq& |S^{d-2}||v|^{-1}\int_0^{\infty}r^{d-{5/2}}(1+r)^{{3/2}|s|+{3/2}\beta}\exp(-\frac{r^2}{8})\int_{-\infty}^{\infty}\exp(-3\frac{y^2}{16})dy dr\\\nonumber
&\leq &C|v|^{-1},
\end{eqnarray*}
where $C$ is a positive constant. Notice that since $\beta> 0$, the integral
$$\int_0^{\infty}r^{d-{5/2}}(1+r)^{{3/2}|s|+{3/2}\beta}\exp(-\frac{r^2}{8})dr,$$
is well-defined.
\\ For the case $d=2$, we perform the same change of variables
\begin{eqnarray*}
J_1&\leq& |S^{d-2}||v|^{-1}\int_0^{\infty}r^{d-{5/2}}(1+r)^{{3/2}|s|+{3/2}\beta}\exp(-\frac{r^2}{8})\\\nonumber
& &\times\int_{r-2|v|}^{r+2|v|}\exp(-\frac{3y^2}{16})(1-(\frac{y-r}{2|v|})^2)^{-1/2}dy dr\\\nonumber
&\leq& C\int_0^{\infty}r^{d-{5/2}}(1+r)^{{3/2}|s|+{3/2}\beta}\exp(-\frac{r^2}{8})\\\nonumber
& &\times\int_{r-2|v|}^{r+2|v|}\exp(-\frac{3y^2}{16})(4|v|^2-(y-r)^2)^{-1/2}dy dr,
\end{eqnarray*}
where $C$ is some positive constant.
\\ We consider the integral in two regions $|y-r|\leq |v|$ and $|y-r|\geq |v|$. On the first region, $(4|v|^2-(y-r)^2)^{-1/2}\leq |v|^{-1}$. On the second region, either $r\geq |v|/2$ or $|y|\geq |v|/2$ gives an exponential decay. Finally, we get
$$\int_{\mathbb{R}^d}K(1+|v'|)^sdv'\leq C(1+|v|)^{\beta-2/3+s}.$$
{\\\it Case 2:} $\beta< 0$.
\begin{eqnarray*}\label{BComponent3KsBetaN}
& &\int_{\mathbb{R}^d}K(1+|v'|)^s dv'\\\nonumber
&\leq & (1+|v|)^s\min\{TC(\epsilon),\frac{C}{\epsilon}\}\left(\int_{\mathbb{R}^d}\exp(-\frac{|v'-v|^2}{8}-\frac{3|v'-v+2(v.\omega)\omega|^2}{16})\right.\\\nonumber
& &\left. \times C|v'-v|^{-{3/2}}(1+|v-(v.\omega)\omega|)^{{3/2}\beta}(1+|v'-v|)^{{3/2}|s|+{3/2}|\beta|}dv'\right)^{2/3}.
\end{eqnarray*}
Again, perform the change of variables $u=r\omega$, $r\in\mathbb{R}_+$, $\omega\in\mathbb{S}^{d-1}$, choose $v$ as the north pole vector in the angle parametrization. 
Denote 
\begin{eqnarray*}
J_2&=&|S^{d-2}|\int_0^{\infty}r^{d-{5/2}}(1+r)^{{3/2}|s|+{3/2}|\beta|}\exp(-\frac{r^2}{8})\times\\
& &\times\int_0^\pi(1+|v|\sin(\varphi))^{{3/2}\beta}\exp(-\frac{3(r+2|v|\cos\varphi)^2}{16})\sin^{d-2}(\varphi)d\varphi dr.
\end{eqnarray*}
Split the integral into two region $|\cos(\varphi)|\leq\frac{1}{\sqrt{2}}$ and $|\cos(\varphi)|>\frac{1}{\sqrt{2}}$. In the first case, since $\sin(\varphi)\geq\frac{1}{\sqrt{2}}$, then
$$(1+|v|\sin(\varphi))^{{3/2}\beta}\leq C(1+|v|)^{{3/2}\beta},$$
the proof is then similar as in the case $\beta>0$. In the second case,
$$\frac{(r+2|v|\cos(\varphi)^2)}{2}\geq \frac{|v|^2}{12}-\frac{r^2}{16},$$
this leads to an exponential decay in $v$.
Finally, we get
$$\int_{\mathbb{R}^d}K(1+|v'|)^sdv'\leq C(1+|v|)^{\beta-2/3+s}.$$
Combine this estimate with $(\ref{BComponent3pre})$, $(\ref{BComponent3prepre})$, $(\ref{BComponent3K})$ and $(\ref{BComponent3Ks})$ to get
\begin{eqnarray*}
I&\leq& \int_{\mathbb{R}^d\times\mathbb{R}^d}|A_n(v)||A_n(v')|Kdvdv'\\
&\leq& \left(\int_{\mathbb{R}^d}|A_n(v)|^2(1+|v|)^{\beta -2/3}\right)^{1/2}[\int_{\mathbb{R}^d}(1+|v|)^{-\beta+2/3}\int_{\mathbb{R}^d}K(v,v')dv'\times\\
& &(\int_{\mathbb{R}^d}K(v,v'')|A_n(v'')|^2dv'')dv]^{1/2},
\end{eqnarray*}
which implies
\begin{equation}\label{BComponent3}
I\leq C\min\{TC(\epsilon),\frac{C}{\epsilon}\}\int_{\mathbb{R}^d}|A_n(v)|^2(1+|v|)^{\beta-2/3}.
\end{equation}
When $\epsilon$ is small and $T$ is large enough, $(\ref{BCollisionKernelFourier})$, $(\ref{BComponent1})$, $(\ref{BComponent2})$, $(\ref{BComponent3})$ imply
\begin{eqnarray}\label{BObservability}
& &\int_0^{T}\int_{\mathbb{T}^d\times\mathbb{R}^d\times\mathbb{R}^d\times\mathbb{S}^{d-1}}\mathcal{B}(|v_*-v|,\omega)\mu_*\mu\times\\\nonumber
& &\times[g'_*{\mu'_*}^{-1/2}+g'\mu'^{-1/2}-g\mu^{-1/2}-g_*\mu^{-1/2}_*]^2d\omega dv_*dvdx\\\nonumber
&\geq &CT\sum_{n\in\mathbb{Z}^d}\int_{\mathbb{R}^d}(1+|v|)^\alpha |A_n(v)|^2dv\\\nonumber
&\geq & CT\int_{\mathbb{T}^d\times\mathbb{R}^d}|g_0|^2dxdv.
\end{eqnarray}
\end{proof}
\begin{proposition}\label{BPropObserve}
Suppose that $\alpha,\beta>0 $ and there exist positive numbers $T_1$ and  $C$ such that
\begin{eqnarray}\label{BoltzmannObservability}
& &\int_0^{T_1}\int_{\mathbb{T}^d\times\mathbb{R}^d\times\mathbb{R}^d\times\mathbb{S}^{d-1}}\mathcal{B}(|v_*-v|,\omega)\mu_*\mu\times\\\nonumber
& &\times[g'_*{\mu'_*}^{-1/2}+g'\mu'^{-1/2}-g\mu^{-1/2}-g_*\mu^{-1/2}_*]^2d\sigma dv_*dvdx\\\nonumber
& \geq &C\int_{\mathbb{T}^d\times\mathbb{R}^d}(|v|+1)^\alpha |f_0|^2dxdv.
\end{eqnarray}
then  there exist positive numbers $T_0$ and $\delta$ such that $\forall$ $t\geq T_0$, $\forall$ $f_0\in L^{\infty}(\mathbb{T}^d\times\mathbb{R}^d)\cap L^{\infty}(\mathbb{R}^d,H^1(\mathbb{T}^d))$
\begin{equation}\label{BoltzmannExponentialDecay}
H_f(t)\leq\exp(-\delta t) H_f(0),
\end{equation}
\end{proposition}
\begin{proof} We check that $L$ satisfies the conditions $(\ref{Lemma3D1})$, $(\ref{Lemma3D2})$, $(\ref{Lemma3D3})$, $(\ref{Lemma3D4})$. Let $\epsilon$ be any positive constant, define 
$$I_\epsilon:=\chi\left(|v-v_*|\leq \frac{1}{\epsilon}\right),$$
\begin{eqnarray*}\label{BoltzmannObservability}
L_{\epsilon,1}[g]
&:=&-\int_0^{T_1}\int_{\mathbb{T}^d\times\mathbb{R}^d\times\mathbb{S}^{d-1}}I_\epsilon\mathcal{B}(|v_*-v|,\omega)\mu_*\mu^{1/2}\times\\\nonumber
& &\times[g'_*{\mu'_*}^{-1/2}+g'\mu'^{-1/2}-g\mu^{-1/2}-g_*\mu^{-1/2}_*]d\sigma dv_*dx,
\end{eqnarray*}
\begin{eqnarray*}\label{BoltzmannObservability}
L_{\epsilon,2}[g]
&:=&-\int_0^{T_1}\int_{\mathbb{T}^d\times\mathbb{R}^d\times\mathbb{S}^{d-1}}(1-I_\epsilon)\mathcal{B}(|v_*-v|,\omega)\mu_*\mu^{1/2}\times\\\nonumber
& &\times[g'_*{\mu'_*}^{-1/2}+g'\mu'^{-1/2}-g\mu^{-1/2}-g_*\mu^{-1/2}_*]d\sigma dv_*dx.
\end{eqnarray*}
It is not difficult to see that $L$, $L_{\epsilon,1}$, $L_{\epsilon,2}$ satisfy $(\ref{Lemma3D1})$, $(\ref{Lemma3D2})$, $(\ref{Lemma3D3})$, with $H'=L^2((1+|v|)^\alpha)$. Proceed similar as in the previous proposition to get
$$\|L_{\epsilon,2}[g]\|_{L^2}^2\leq C(\epsilon)\int_{\mathbb{T}^d\times\mathbb{R}^d}(|v|+1)^\beta|g|^2dxdv,$$
which means that $(\ref{Lemma3D4})$ is satisfied. By Lemma $\ref{Lemma3}$, the conclusion of the proposition follows.
\end{proof}
\subsection{Convergence to Equilibrium: Proof of Theorem $\ref{TheoremLinearizedBoltzmannDecay}$}
The case $\alpha,\beta> 0$ is straight forward from Proposition $\ref{BPropObservability}$ and Proposition $\ref{BPropObserve}$. We now prove the theorem for the case $-d+1<\alpha,\beta<0$.
According to Proposition $\ref{BPropObservability}$ and Lemma $\ref{Lemma1}$, there exist a time $T$ and a constant $C$ such that
\begin{eqnarray*}\label{BPropPolynomial1}
\|f(0)\|_{L^2}^2-\|f(T)\|_{L^2}^2\geq C\int_0^T\int_{\mathbb{T}^d\times\mathbb{R}^d}(|v|+1)^\alpha|f(0)|^2dxdv.
\end{eqnarray*}
This implies that for all $k$ in $\mathbb{d}$
\begin{equation}\label{BPropPolynomial1}
\|f(kT)\|_{L^2}^2-\|f((k+1)T)\|_{L^2}^2\geq C\int_0^T\int_{\mathbb{T}^d\times\mathbb{R}^d}(|v|+1)^\alpha|f(kT)|^2dxdv.
\end{equation}
Now, for positive numbers $k_1$, $k_2$ and $k_3$ satisfying $\alpha k_1+k_2k_3>0$, according to the Holder inequality 
\begin{equation}\label{BPropPolynomial2}
\left(\int_{\mathbb{T}^d\times\mathbb{R}^d}(|v|+1)^\alpha |f(kT)|^2\right)^{k_1}\left(\int_{\mathbb{T}^d\times\mathbb{R}^d}(|v|+1)^{k_2} |f(kT)|^2\right)^{k_3}\geq \left(\int_{\mathbb{T}^d\times\mathbb{R}^d}|f(kT)|^2\right)^{k_1+k_3}.
\end{equation}
Combine $(\ref{BPropPolynomial1})$ and $(\ref{BPropPolynomial2})$ to get
\begin{equation}\label{BPropPolynomial3}
\|f((k+1)T)\|_{L^2}^2\leq \|f(kT)\|_{L^2}^2-C\frac{\|f(kT)\|_{L^2}^{2\frac{k_1+k_3}{k_1}}}{\left(\int_{\mathbb{T}^d\times\mathbb{R}^d}(|v|+1)^{k_2} |f(kT)|^2\right)^{\frac{k_3}{k_1}}}.
\end{equation}
Now, choose $(|v|+1)^{k_2}f$ as a test function in the linearized Boltzmann equation, for $k_2$ small enough, as $\beta<0$ there exists a negative number $k'_2$ depending on $k_2$ such that
\begin{eqnarray*}
\|(|v|+1)^{k_2/2}f(kT)\|_{L^2}^2-\|(|v|+1)^{k_2/2}f(0)\|_{L^2}^2&\leq&\int_0^{kT}\|(|v|+1)^{k_2'/2}f(t)\|_{L^2}^2dt\\\nonumber
&\leq & {kTC}\|f(0)\|_{L^2}^2,
\end{eqnarray*}
then 
\begin{equation}\label{BPropPolynomial4}
\|(|v|+1)^{k_2/2}f(kT)\|_{L^2}^2 \leq  {kTC}\|f(0)\|_{L^2((1+|v|)^{k_2/2})}^2,
\end{equation}
where $C$ is some positive constant.
The two inequalities $(\ref{BPropPolynomial3})$ and $(\ref{BPropPolynomial4})$ lead to
\begin{eqnarray*}
\|f((k+1)T)\|_{L^2}^2\leq \|f(kT)\|_{L^2}^2-C\frac{\|f(kT)\|_{L^2}^{2\frac{k_1+k_3}{k_1}}}{k^{\frac{k_3}{k_1}}}.
\end{eqnarray*}
This implies
\begin{eqnarray*}
\frac{\|f((k+1)T)\|_{L^2}^2}{k+1}\leq \frac{\|f(kT)\|_{L^2}^2}{k}-C\left(\frac{\|f(kT)\|_{L^2}^2}{k}\right)^{\frac{k_1+k_3}{k_1}}.
\end{eqnarray*}
According to Lemma $\ref{LemmaAmmari}$,
$$\|f(kT)\|^2\leq \frac{Mk}{(k+1)^{\frac{k_1}{k_3}}}.$$
Let $\frac{k_1}{k_3}>-\frac{\alpha}{k_2}$ tend to $\infty$, we get the theorem.
\section{Conclusion} We have presented a new approach to the problem of convergence to equilibrium of kinetic equations. The idea of our technique is to prove a 'weak' coercive estimate on the damping. The approach seems to work very well in the context of linear equations. An reasonable question is if this technique could be extended to study the trend to equilibrium of nonlinear kinetic equations, where a typical example is the nonlinear Boltzmann equation. In an ongoing project, we are trying to analyse this method for the linearized Uehling-Uhlenbeck equation, where a spectral gap estimate is hard to obtain but a 'weak' coercive estimate is easier to get. 
\\ {\bf Acknowledgements.} The author would like to thank his advisor, Professor Enrique Zuazua for suggesting this research topic, for his great guidance, his wise advices and his constant encouragement while supervising this work. He would also like to thank Professor Belhassen Dehman for fruitful discussions. The author has been partially supported by the ERC Advanced Grant FP7-246775 NUMERIWAVES. 
\bibliographystyle{plain}\bibliography{ConvergenceToEquilibriumVersion1}
\end{document}